\documentclass[11pt]{amsart}
\usepackage[latin9]{inputenc}
\usepackage{geometry}
\usepackage{color}
\definecolor{note_fontcolor}{rgb}{0.800781, 0.800781, 0.800781} 
\usepackage{amsthm}
\usepackage{amstext}
\usepackage{esint}

\makeatletter

\providecommand{\tabularnewline}{\\}

\numberwithin{equation}{section}
\numberwithin{figure}{section}
\theoremstyle{plain}
\newtheorem{thm}{\protect\theoremname}
  \theoremstyle{plain}
  \newtheorem{conjecture}[thm]{\protect\conjecturename}
  \theoremstyle{definition}
  \newtheorem{defn}[thm]{\protect\definitionname}
  \theoremstyle{plain}
  \newtheorem{lem}[thm]{\protect\lemmaname}

  \theoremstyle{plain}
  \newtheorem{prop}[thm]{\protect\propositionname}
  \theoremstyle{remark}
  
\usepackage{verbatim}
\usepackage{tikz}
\usetikzlibrary{calc,fadings,decorations.pathreplacing}
\usepackage{verbatim}
\usepackage{graphicx}

\theoremstyle{plain}

\newtheorem{coro}[thm]{Corollary} 
\long\def\symbolfootnote[#1]#2{\begingroup%
\def\thefootnote{\fnsymbol{footnote}}\footnote[#1]{#2}\endgroup}   
 
\theoremstyle{remark}

\thanks{Bourgain is supported in part by NSF grant DMS-1301619. Kontorovich is supported in part by an NSF CAREER grant DMS-1254788 and DMS-1455705, an NSF FRG grant DMS-1463940, and Alfred P. Sloan Research Fellowship, and a BSF grant. Magee was supported in part by NSF Grant DMS-1128155. Oh was supported in part by NSF Grant DMS-1361673.}

\author{Michael Magee, Hee Oh and Dale Winter\\
With an Appendix by \\Jean Bourgain, Alex Kontorovich and Michael Magee}
\address{Institute for Advanced Study, Princeton, NJ 08540}
\email{bourgain@math.ias.edu}
\address{Rutgers University, New Brunswick, NJ}
\email{alex.kontorovich@rutgers.edu}
\address{Mathematics department, Yale university, New Haven, CT 06511}
\email{michael.magee@yale.edu}
\address{Mathematics department, Yale university, New Haven, CT 06511 and Korea Institute for Advanced Study, Seoul, Korea}
\email{hee.oh@yale.edu}
\address{Institute for Advanced Study, Princeton, NJ 08540}
\email{dale.alan.winter@gmail.com}

\makeatother

  \providecommand{\conjecturename}{Conjecture}
  \providecommand{\definitionname}{Definition}
  \providecommand{\lemmaname}{Lemma}
  \providecommand{\propositionname}{Proposition}
\providecommand{\theoremname}{Theorem}

\begin{document}
\global\long\def\nat{\mathbf{N}}
 \global\long\def\R{\mathbf{R}}
 \global\long\def\C{\mathbf{C}}
 \global\long\def\Z{\mathbf{Z}}
 \global\long\def\Q{\mathcal{Q}}
\global\long\def\RR{\mathcal{R}}

\global\long\def\a{\alpha}
 \global\long\def\b{\beta}
 \global\long\def\g{\gamma}
 \global\long\def\d{\delta}
 \global\long\def\e{\epsilon}
 \global\long\def\i{\iota}
 \global\long\def\G{\Gamma}
 \global\long\def\GG{\mathbf{G}}
 \global\long\def\vp{\varphi}

\global\long\def\GL{\mathrm{GL}}

\global\long\def\A{\mathcal{A}}
\global\long\def\AA{\mathbb{A}}
 \global\long\def\B{\mathcal{B}}
 \global\long\def\E{\mathcal{E}}
 \global\long\def\H{\mathbb{H}}
 \global\long\def\HH{\mathcal{H}}
 \global\long\def\N{\mathcal{N}}
 \global\long\def\O{\mathcal{O}}
 \global\long\def\Ohat{\widehat{\O}}
 \global\long\def\P{\mathcal{P}}
 \global\long\def\K{\mathcal{K}}
 \global\long\def\k{\kappa}
 \global\long\def\T{\mathbb{T}}
 \global\long\def\s{\sigma}
 \global\long\def\spec{\mathrm{spec}}
 \global\long\def\SO{\mathrm{SO}}
 \global\long\def\End{\mathrm{End}}
 \global\long\def\so{\mathfrak{so}}
 \global\long\def\SL{\mathrm{SL}}
 \global\long\def\CC{\mathcal{C}}
 \global\long\def\vol{\mathrm{vol}}
 \global\long\def\Ind{\mathrm{Ind}}

\global\long\def\F{\mathbb{F}}
 \global\long\def\D{\mathcal{D}}
 \global\long\def\L{\mathcal{L}}
 \global\long\def\Mhat{\hat{\mathcal{M}}}
 \global\long\def\diam{\mathrm{diam}}
 \global\long\def\U{\mathcal{U}}
 \global\long\def\Int{\mathrm{Int}}
 \global\long\def\I{\mathcal{I}}
 \global\long\def\Xhat{\widehat{X}}
 \global\long\def\NN{\mathbb{N}}
 \global\long\def\Lip{ {C^{1}}}
 \global\long\def\ev{\mathrm{ev}}
 \global\long\def\tr{\mathrm{tr}}
\global\long\def\K{K}
\global\long\def\br{\mathbf R}

\title{Uniform congruence counting for Schottky semigroups in $\SL_{2}(\Z)$}
\begin{abstract} Let $\Gamma$ be  a Schottky semigroup in $\SL_2(\Z)$,
and for $q\in \mathbf N$, let 
$\Gamma(q):=\{\gamma\in \Gamma: \gamma= e \text{ (mod $q$)}\}$ be its congruence subsemigroup
of level $q$. Let $\delta$ denote the Hausdorff dimension of the limit set of $\Gamma$.
 We prove the following uniform congruence counting theorem
with respect to the family of Euclidean norm balls $B_R$ in $M_2(\R)$ of radius $R$:
  for all positive integer $q$ with no small prime factors, 
 $$\# (\Gamma (q)  \cap B_R )= c_\Gamma \frac{R^{2\delta}}{ \# (\SL_2(\Z/q\Z))} +O(q^C R^{2\delta -\epsilon})$$
 as $R\to \infty$ for some $c_\Gamma >0, C>0, \e>0$ which are independent of $q$.
Our technique also applies to give a similar counting result for the continued fractions semigroup of $\SL_2(\Z)$,
which arises in the study of Zaremba's conjecture on continued fractions. 
\end{abstract}

\maketitle

\section{Introduction} Let $\SL_2(\R)$ act on $\R\cup\{\infty\}$ by M\"obius transformations.
We say that the collection of elements $g_1, \ldots, g_k\in \SL_2(\R)$, $k\ge 2$, is  a Schottky generating set
if there exist mutually disjoint compact intervals $I_1, \ldots, I_{k}, J_1, \ldots, J_k$ in $\R$ such that
$g_i$ maps the exterior of $J_i$ onto the interior of $I_{i}$ for each $1\le i\le k$.
\markleft{}
We call a semigroup $\Gamma\subset$
$\SL_{2}(\R)$ Schottky if it is generated by some Schottky generating set as a semigroup. By the ping-pong argument,
Schottky semigroups are necessarily discrete and free.
Schottky semigroups are ubiquitous in $\SL_2(\R)$; for instance, for any hyperbolic elements
$h_1, h_2\in \SL_2(\br)$ with no common fixed points on  $\R\cup\{\infty\}$, the pair $h_1^m,h_2^m$
forms a Schottky generating set for all sufficiently large $m$. 

When $\Gamma $ is a semigroup in $\SL_2(\Z)$ and $q\in \mathbf N$, the congruence subsemigroup 
of $\Gamma$ of level $q$
is defined by
$$\Gamma(q):=\{\gamma\in \Gamma: \gamma=e \text{ mod $q$}\} .$$

The main aim of this paper is to study a congruence lattice point counting problem for $\Gamma(q)$ in a Schottky semigroup $\Gamma\subset \SL_2(\mathbf Z)$ with a uniform power-savings error term.
 For $R>0$, consider the ball of radius $R$ with respect to the Frobenius norm: \[ B_R:=
\Big\{ \left(\begin{array}{cc}
a & b\\
c & d
\end{array}\right) \in \SL_2(\R) : \sqrt{a^{2}+b^{2}+c^{2}+d^{2}} <R \Big\}.
\]
The following is a simplified version of our main theorem
(see Theorem \ref{thm:maintheoremelaborate} for a more refined version):

\begin{thm}
\label{thm:mainsimple}If $\G$ is a Schottky semigroup of $\SL_2(\Z)$, there exist
$Q_{0}\in\mathbf{N}$, $c_{\Gamma}>0,$ $C>0$ and $\e>0$ such that
for all $q\in \mathbf N$ with $(Q_{0},q)=1$, 
\[ \# \Gamma(q) \cap B_R=c_{\Gamma}\frac{R^{2\delta}}{\# \SL_{2}(\Z/q\Z)}+O\left(q^{C}R^{2\delta_{\mathcal{}}-\e}\right)
\]
where $\delta>0$ is the Hausdorff dimension of the limit set of $\Gamma$.
\end{thm}
The limit set of $\Gamma$ is the set of all accumulation points of an orbit $\Gamma .o$
in $\R\cup\{\infty\}$.

\noindent{\bf Remark}
\begin{enumerate}
\item When $\Gamma$ is a Schottky {\it subgroup} of $\SL_2(\Z)$, the analogous
result to Theorem \ref{thm:mainsimple} was proved by Gamburd \cite{GAMBURDGAP} for $\delta>5/6$, by
 Bourgain-Gamburd-Sarnak \cite{BGS2} for $\delta>1/2$ and by Oh-Winter \cite{OW} for any $\delta>0$. The last two results
 are restricted to the moduli condition of $q$ square-free. The counting result of Oh-Winter is deduced from \cite{MO} based on  the uniform
 exponential mixing of the geodesic flow for the congruence covers of a Schottky surface, and
 hence does not apply to the semigroup counting. 
 
\item So the novelty of  Theorem \ref{thm:mainsimple} 
lies in the treatment of a Schottky {\it semigroup} and  the uniformity of the power-savings error term
for {\it all} moduli $q$  (with no small prime factors). The extension to the arbitrary moduli $q$ case
relies on the new technology that appears in the Appendix by Bourgain, Kontorovich
and Magee. 

\item We also remark that for fixed $q$, 
Theorem \ref{thm:mainsimple} follows from the work of Naud \cite{NAUD} in this generality.
 We refer to \cite{BGS2} for more backgrounds on earlier related works. \end{enumerate}

Our methods also apply to a congruence family of semigroups related to continued
fractions and Diophantine approximation. Let $\A$ be a finite set of at least two positive integers.
Define $
\mathcal{{G}}_{A}$ to be the subsemigroup of $\mathrm{GL}_2(\Z)$ generated by
$$ g_a:= \left(\begin{array}{cc}
0 & 1\\
1 & a
\end{array}\right),\quad a\in \A .$$
We define the {\it continued fractions semigroup} $\Gamma_{\A}$ as follows:
 $$\Gamma_{\A}:=\mathcal{\mathcal{G}}_{\A}\cap\mathrm{SL_{2}(\Z)},$$
 in other words, $\G_\A$ is a semigroup generated by $\{g_a g_{a'}: a, a'\in \A\}$.
The continued fractions semigroup $\Gamma_{\A}$ is not  a Schottky semigroup; however
 the methods of proof
of Theorem \ref{thm:mainsimple} apply as well:

\begin{thm}
\label{thm:maincontinued}Theorem \ref{thm:mainsimple} also holds
for the continued fractions semigroup $\G_{\A}$.
\end{thm}

In order to explain the relation of $\G_{\A}$ with continued fractions, we set
\[
[a_{1},\ldots,a_{l},\ldots]:= \cfrac{1}{a_{1}+\cfrac{1}{a_{2}+\ddots\cfrac{1}{a_{l}+\ddots}}}
\]  for any
sequence of $a_{i}\in\mathbf{N}$.

Write 
\[
\mathfrak{R}_{\A}:=\{[a_{1},\ldots,a_{k}]:\:k\in\mathbf{N},a_{i}\in\A\:\text{ for all } i\:\}
\]
for the set of approximants to $\mathfrak{C}_{\A},$ and $\mathfrak{D}_{\A}$
for the set of denominators of reduced elements of $\mathfrak{R}_{\A},$
that is,
\[
\mathfrak{D}_{\A}:=\{d:\:\frac{b}{d}\in\mathfrak{R}_{\A}\text{ for some \ensuremath{b} coprime to \ensuremath{d}\:\}}.
\] 
For an integer $A\in \mathbf N$, we write $\mathfrak{D}_{[A]}= \mathfrak{D}_{\{1, 2, \cdots, A\}}$.
In \cite{ZAREMBA}, Zaremba made the following conjecture,
motivated by applications to numerical analysis.
\begin{conjecture}[Zaremba]
There is some absolute $A\in\mathbf{N}$ such that $\mathfrak{D}_{[A]}=\mathbf{N}$. \end{conjecture}

Observe that
\[
\frac{b}{d}=[a_{1},\ldots,,a_{k}]
\]
if and only if
\[
\left(\begin{array}{cc}
0 & 1\\
1 & a_{1}
\end{array}\right)\left(\begin{array}{cc}
0 & 1\\
1 & a_{2}
\end{array}\right)\ldots\left(\begin{array}{cc}
0 & 1\\
1 & a_{k}
\end{array}\right)=\left(\begin{array}{cc}
\star & b\\
\star & d
\end{array}\right).
\] 
This  yields the relation
$$\mathfrak D_{\A}=\left\{ \langle \gamma (0,1)^t , (0,1)^t  \rangle : \gamma\in \mathcal G_\A\right\}$$
where $\langle \cdot, \cdot \rangle$ denotes the standard inner product on $\R^2$,
thus enters the semi-group $\mathcal G_A$
in the study of continued fractions. 

Bourgain and Kontorovich  \cite[Theorem 1.2]{BKANNALS} proved that 
Zaremba's conjecture is true after
replacing $\mathbf{N}$ by a density one subset. That is, there is
some $A$ such that
\begin{equation}
\# \mathfrak{D}_{[A]}\cap\{1,\ldots,N\}=N+o(N). \label{eq:bk}
\end{equation} 

Furthermore, they showed that the $o(N)$ term
can be taken to be $O(N^{1-c/\log\log N}$) for suitable $c>0$ (this relies on the Appendix) and $A=50$ will suffice. The size of $A$ 
has since been improved to $A=5$ by Huang \cite{HUANG}, following previous innovations by Frolenkov and Kan \cite{FK14} on the necessary $\delta_\A$.


Theorem \ref{thm:maincontinued} provides the precise missing ingredient in Bourgain and Kontorovich's work  \cite{BKANNALS},
to replacing the $o(N)$
bound for the size of the exceptional set in \eqref{eq:bk} with a power savings error $O(N^{1-\e})$. Indeed,
combining Bourgain and Kontorovich's method from \cite{BKANNALS},
Huang's refinement, and with the counting estimate of Theorem \ref{thm:maincontinued}
and its technical form Theorem \ref{thm:maintheoremelaborate} in
place of \cite[Theorem 8.1]{BKANNALS},  one can derive
the following improvement of \eqref{eq:bk}:
for $\mathcal{A}=\{1,2,3,4,5\}$ and for some $\e>0$,
\begin{equation} \label{mainZa-2} 
\# \mathfrak{\mathfrak{D}}_{\A}\cap\{1,\ldots,N\}=N+O(N^{1-\epsilon}).
\end{equation}

The key point is that the uniform lattice point count enables
us to replace the parameter $\mathcal{Q}=N^{\alpha_{0}/\log\log N}$
in \cite{BKANNALS} and \cite{HUANG} with a power of $N$.

We remark that  a short alternative argument for \eqref{mainZa-2} was recently proposed by Bourgain in \cite{BOURGAINPARTIAL}. His argument deviates from the approach of \cite{BKANNALS} and hence does not require orbital counting estimates.


We draw the reader's attention to the survey article \cite{BOURGAINSURVEY}
where other applications to continued fractions are discussed. The reader can also see
the survey of Kontorovich \cite{KONTOROVICH} that situates Zaremba's conjecture amongst other problems in the `thin (semi)groups' setting.

\noindent{\bf Overview of the proofs of Theorems \ref{thm:mainsimple} and \ref{thm:maincontinued}:}

 The basic strategy is to regard our Schottky semigroup setup as an expanding map and to apply transfer operator techniques. Necessary spectral bounds are then deduced by synthesizing work of Bourgain-Varj\'{u}, Bourgain-Gamburd-Sarnak, Dolgopyat, and Naud.  For now we focus on the arguments for Theorem \ref{thm:mainsimple}; those for Theorem \ref{thm:maincontinued} are similar. 

We consider the map $T:I:=\cup_{i=1}^kI_{i}\to \R$ defined by
\begin{equation}
T\lvert_{I_{i}}=(g_{i})^{-1}\label{eq:Tdef}
\end{equation}
and the distortion function $\tau:I\to \R$ given by $\tau(x)=\log |T'(x)|$, which is
 eventually positive in our setting. 
 The transfer operator $\L_s$ is defined for all $s\in \C$ by
 $$\L_s(f)(x)=\sum_{Ty=x} e^{-s\tau (y)} f(y) $$
 as a bounded linear operator on $C^1(I)$. Lalley's renewal equation   \cite{LALLEYSYMB} provides a link between the counting problem for $\G$ and spectral bounds for $\L_s$. 
 Such spectral bounds were obtained by Naud \cite{NAUD}, who provided a $C^1$-operator norm estimate on $\L_s^m$ valid on a strip $|\Re(s)-\delta|<\e$ and so deduced\footnote{Naud uses Ruelle  zeta function techniques as in \cite{RUELLEFREDHOLM}, in contrast to our use of the renewal equation.} the case $q=1$ of Theorem \ref{thm:mainsimple}.


To provide a counting result that is uniformly accurate over congruence semigroups we must actually deal with the congruence transfer operators.  More precisely,
let $c_q:I\to \SL_2( \Z/q\Z)$ be the cocycle given by
$$c_q|_{I_i}= g_i \text{ mod $q$} ,$$
and define the congruence transfer operator $$\L_{s,q} [F](x)=\sum_{Ty=x} e^{-s \tau y}
c_q(y).F(y)$$
on the space of $\C^{\G_q}$-valued functions for $ \G_q:=\SL_2( \Z/q\Z)$. The composition $c_q(y).F(y)$ is the result of applying  $c_q(y)\in \Gamma_q$ to the vector $F(y)\in \C^{\G_q}$ by the right regular representation of $\Gamma_q$. It is also useful throughout the paper to think of $F$ as a function on $I\times \C^{\Gamma_q}$. We fix the standard Hermitian form on $\C^{\Gamma_q}$ that comes from the identification of $\Gamma_q$ with the standard basis of $\C^\Gamma_q$ and defining $\langle g_1, g_2 \rangle = \delta_{g_1,g_2}$. The space $\C^{\Gamma_q}\ominus 1$ is defined to be the space of functions that are orthogonal to constants with respect to the fixed Hermitian form. The vector space $\C^{\Gamma_q}\ominus 1$ inherits a Hermitian form from that of $\C^{\Gamma_q}$. It is with respect to this form that  we define the Banach spaces $C^{1}(I;\C^{\G_{q}}\ominus1)$ that play a central role in this paper.

The following is the main technical result:
\begin{thm}[Bounds for congruence transfer operators]
\label{thm:maintransfer}Write $s=a+ib$. There is $Q_{0}\in\nat$
such that for any $\eta>0$, there are $\e=\e(\eta)>0$, $b_{0}>0$,
$0<\rho_{\eta}<1$, $C_{\eta}>0$, $0<\rho_{0}<1$, $r>0$ and $C>0$
such that the following holds for all $a\in\R$ with $|a-\delta|<\e$
and $b\in\R$: 
\begin{enumerate}
\item \label{enu:mainmodular} When $|b|\leq b_{0}$ and $f\in C^{1}(I;\C^{\G_{q}}\ominus1)$
\[
\|\L_{s,q}^{m}f\|_{C^{1}}\leq Cq^{C}\rho_{0}^{m}\|f\|_{C^{1}}
\]
when $(q,Q_{0})=1$. Here $\C^{\G_{q}}\ominus1$ is the orthogonal complement
to the constant functions in the right regular representation
of $\G_{q}$. 
\item \label{enu:mainlargeimaginary} When $|b|>b_{0}$ 
\[
\|\L_{s,q}^{m}\|_{C^{1}}\leq C_{\eta}|b|^{1+\eta}\rho_{\eta}^{m}
\]
uniformly with respect to $q\in\nat$. 
\end{enumerate}
\end{thm}
The transfer operators have two parameters $s$, the Laplace transform-dual/frequency
version of the counting parameter, and $q$, the modular parameter.
Since inverting the Laplace transform that was taken involves an infinite
vertical contour, one must obtain spectral bounds that are uniform
in $s$ with $\Re(s)$ within some fixed small window of $\delta$.
The bounds should also be uniform with respect to the currently considered
family of moduli $q.$ These bounds rely on two different inputs that
both involve deep ideas.

To address large imaginary part of $s$ considerations, we will use the
method of Dolgopyat from \cite{DOLG}, and its further development
by Naud from \cite{NAUD}. We follow Naud's analysis from \cite{NAUD} up to the point of departure from Naud's
work in Lemma \ref{technical} where we extend \cite[Lemma 5.10]{NAUD}
to vector valued functions. Here, an important point that prevents the cocycle
$c_{q}$ from interfering with the non-stationary phase is that it
is locally constant. We mention that this observation was first due to \cite{OW}
where they consider the congruence transfer operator associated to the Markov partition
given by the geodesic flow.

For bounded $\Im(s)$ and varying $q$ we follow the work of Bourgain,
Gamburd and Sarnak from \cite{BGS2} and the work
of Bourgain, Kontorovich and Magee in the Appendix, which allows us to relate
the norm $\|\L_{s,q}^m\|_{C^1}$ to the expander result on the Cayley graphs of the
$\G_q$ with respect to a fixed generating set of $g_i$'s. The main reason behind
our successful treatment of arbitrary moduli $q$ case is  the work of Bourgain-Varj\'{u}
establishing the expander result for $\SL_2(\Z/q\Z)$ for arbitrary $q$, as explained
in the Appendix.

\subsection{Acknowledgements}

We would like to thank Peter Sarnak for his encouragement and support
throughout this project. We thank Jean Bourgain, Alex Kontorovich
and Curt McMullen for helpful comments on an earlier version of this
paper.

\section{Dynamics and Thermodynamics on the boundary}
\label{continuedfractions}

\subsection{The dynamical system $T$.\label{sub:A-dynamical-system}}

We construct a dynamical system $T: I\to \R$ on a disjoint union of intervals $I$ that plays a central role in the counting
estimates of our main Theorems \ref{thm:mainsimple} and \ref{thm:maincontinued}, and set up the  notations and the assumptions
 which will be used
throughout the paper.

\noindent{\bf I: Schottky semigroup case:}

 Let $g_1, \cdots, g_{k'}$ ($k'\ge 2$) be 
the Schottky generating set in $\SL_2(\Z)$. We let $\{\tilde I_{i}, \tilde J_i:i=1,\ldots,k'\}$ be the intervals  such that
 $g_i$ maps the exterior of $\tilde J_i$ onto the interior of $\tilde I_i$ as in the definition of the Schottky generators.
 Set $g_{k'+\ell}=g_\ell^{-1}$ and $\tilde I_{k'+\ell}=\tilde J_\ell$ for $1 \le \ell \le k'$.

For any $0\le \ell \le k'$, let $\G$ be
the semigroup generated by $g_1, \cdots, g_{k'}, g_{k'+1}, \cdots, g_{k'+\ell}$; we will call $\Gamma$ a Schottky semigroup.
This is slightly more general than the definition we gave in the Introduction, and the main reason of this extension
is to include Schottky groups in our discussion of Schottky semigroups.
Note that when $\ell =k'$, $\G$ coincides with the Schottky {\it subgroup} generated by $g_1,\cdots, g_{k'}$. 

Set $p=k'+\ell$. We
now define a map $B:\tilde I \to \R\cup\{\infty\}$ for $\tilde I:=\cup_{i=1}^p \tilde I_i  $ by the piecewise M\"obius action
\begin{equation}
B |_{\tilde I_{i}}=g_{i}^{-1}.
\end{equation}
Since $g_i(\infty)\in \tilde I_i$, the image of $B$ contains $\infty$.

The \emph{cylinders }of length $n$ are by definition the
sets of the form \begin{equation}
\tilde I_{i_1}\cap B^{-1}(\tilde I_{i_2})\cap \cdots \cap B^{-(n-1)}( \tilde I_{i_{n}})
\label{eq:cylinder}
\end{equation}
where each $1\leq i_{j}\leq p.$ Let $I$ be the union of the cylinders of length 2 and
define
$$T:I \to \br$$
to be the  restriction of $B$ to $I$.  Note that
$g_i(\infty)\notin I$  and hence the image of $T$ does not contain $\infty$; it is for this reason that we replaced $\tilde I$ with $I$. 
Finally, we say that a sequence $g_{i_1},g_{i_2},g_{i_3},\ldots$ of the Schottky generators is \emph{admissible} if no $g_{i_j}$  is 
followed by its inverse. This means all the words obtained by concatenating consecutive subsequences are reduced. 
We now let $k$ denote the number of cylinders of length 2.

\noindent{\bf II: Continued fraction semigroup case:}
Let $\A$ be a finite subset of $\mathbf{N}$ with at least two
elements. For $a\in \A$, set
$$ g_a:= \left(\begin{array}{cc}
0 & 1\\
1 & a
\end{array}\right).$$
Let $\G$ be the continued fractions semigroup $\G_\A$ generated by $g_ag_{a'}$, $a, a'\in \A$.
Since $a, a' \ge 1$, it follows that the trace of any element of $\G$ is strictly bigger than $2$
and hence every element of $\G$ is hyperbolic.

Note that the $g_{a}$ acts as M\"obius transformations on $\R\cup\{\infty\}$ by 
\[
g_{a}(z)=\frac{1}{z+a}.
\]
Let $A$ denote the largest member of $\A$ and consider the interval $I_A:= \left[\frac{1}{A+1},1\right]$.

 For $a\in\A$, let $ I_{a}:=g_a I_A$, which can be computed to be
\[
I_{a}=\left[\frac{1}{a+1},\frac{1}{a+(A+1)^{-1}}\right]\subset\left[\frac{1}{a+1},\frac{1}{a}\right].
\]
The $I_{a}$ are clearly disjoint as $A\geq1$. It follows that $g_a$'s  generate a free semigroup by the ping-pong argument.
We also record for later use that the derivative of the M\"obius action has
\begin{equation} \label{derivativebound}
g'_{a}(z)=\frac{1}{|z+a|^{2}}\leq(a+(1+A)^{-1})^{-2}\leq(1+(1+A)^{-1})^{-2}
\end{equation} 
for all $z\in I_A$. 
We now set \[
I_{a,a'}:=g_{a}g_{a'} I_A \subset I_{a}
\]
obtaining a disjoint collection of $\#\A^2$ number of
closed intervals. Rename these intervals $I_{a,a'}$ and corresponding elements $g_ag_a'$
as $I_i$'s and $g_i$'s respectively.

Define
\[
T:I\to\R,\quad T\lvert_{I_i}=(g_i)^{-1}.\]
Note that $g_{a}g_{a'}I\subset I_{a,a'}$, in other words, $g_i I\subset I_i$ for each $1\le i\le  \# \A^2$. 
Again, we let $k = \# \A^2$ denote the number of intervals obtained.

\medskip

\noindent{\bf Set-up:}  In the rest of this paper,
let $\Gamma$ be a Schottky semigroup or  the continued fractions semigroup, with the associated locally analytic map
$$
T:I=\cup_{i}I_{i}\to\R \quad\text{
given by $T|_{I_i}=g_i^{-1}$}$$
constructed above.

It follows easily from the construction that we have the
\begin{description}
\item [{Markov property}] If $T(I_{i})\cap I_{j}\neq0$ then $T(I_{i})\supset I_{j}.$
\end{description}
\begin{prop}
\label{def:expansion-1} The map $T$ is eventually expanding, that is,  there
are $D>0,\g>1$ such that for all $N\geq1$ and $x\in T^{-N+1}(I)$
\[
|(T^{N})'(x)|\geq D^{-1}\g^{N},
\]
wherever the derivative exists\footnote{The derivative may have poles.}
in $T^{-N+1}(I)$.
\end{prop}
\begin{proof} For the Schottky semigroup case,  this can be proved exactly as in the
proof of  \cite[Proposition 15.4]{BORTHWICK}. For the continued fraction case it follows from \eqref{derivativebound}
and the chain rule that for any $z\in I$, \[
|T'(z)|\geq(1+(1+A)^{-1})^{4}>1
\] and hence the claim follows.
\end{proof}

We also must introduce the
following \emph{distortion function }on $I$.
\begin{defn}[Distortion function]
\label{def:The-distortion-function}The distortion function $\hat \tau:I\to\R$
is defined by

\[
\hat \tau(x)=\log|T'(x)|.
\]

\end{defn}

This definition is very natural for our purposes. For certain technical calculations, however, it is easier to work with a slightly different version. We consider the Cayley map $J$ from the upper half plane to the unit disc sending $i$ to the center $0$ of the disc. We can therefore think of $T$ as acting on the the subset $J(I)$ of the unit circle. This gives an alternative distortion function.

\begin{defn}[Distortion function II]
\label{def:The-distortion-function-II}The distortion function $ \tau:I\to\R$
is defined by

\[
 \tau(x)=\log|(J\circ T\circ J^{-1})'(Jx)|.
\]

\end{defn}

The two distortion functions mentioned here are cohomologous (that is, there is a function $h = - \log (J')$ such that $\hat{\tau}(x) = \tau(x) + h(x) - h(T(x))$), so are equivalent for many purposes. Sometimes it is convenient to work with one, sometimes the other.

Since  $T$ is real analytic, and it is easy to see that $T'$ is never zero
on $I$, it follows that $\tau$ is real analytic on $I.$  The iterated version
\[
\tau^{N}(x):=\sum_{i=0}^{N-1}\tau(T^{i}x)
\]
measures the distortion along a trajectory of $T.$ It follows from
the eventually expanding property of $T$ that there is an $N_{0}$
such that for all $N\geq N_{0},$ $\hat \tau^{N}>0$ on the cylinders of
length $N$, that is, $\hat \tau$
is \emph{eventually positive.} Since $\tau$ is cohomologous to $\hat \tau$ we conclude that $\tau$ is also eventually positive.

 Let $d_{E}$ denote Euclidean distance in the upper half plane, and fix the basepoint $o = i\in \mathbb{H}$. The following Lemma links the lattice point count with the dynamical system
we have defined.

\begin{lem}
\label{contracting}There exist $C,r>0$ and $\k<1$ such that if 
 $k_{0}$ is a point in $I$ then for $L\in\mathbf{N}$ and admissible sequence of $g_{i_j}$
\begin{equation}
d_{E}(g_{i_{1}\ldots}g_{i_{L}}o,g_{i_{1}\ldots}g_{i_{L}}k_{0})\leq C\k^{L}.\label{eq:contracting}
\end{equation}
If in the general Schottky semigroup case, we also require that   $k_0\notin\tilde{I}_i$, where $i=i_L+k'\bmod 2k'$.

 \end{lem}
\begin{proof}
The inequality \eqref{eq:contracting} follows from the fact that
M\"obius transformations preserve (generalized) circles orthogonal to
the boundary of $\H$, together with the eventually expanding property
of $T$. \end{proof}

We denote by $K$ the limit set of $\Gamma$, i.e., the set of
all accumulation points in $\partial(\mathbb H)=\br \cup\{\infty\}$ of the orbit $\G. o$.  It follows from Lemma \ref{contracting} that the limit set $K$  is also
given by the $T$-invariant set

\[
\K=\bigcap_{i=1}^{\infty}T^{-i}(I).
\]

In order to perform counting in congruence classes, we need to twist
our dynamical system by a family of locally constant maps. Let $\G_{q}=\SL_{2}(\Z/q\Z).$
\begin{defn}[Modular cocycle]
\label{def:modcocycle}

For every modulus $q\in\mathbf{N}$, define 
$c_{q}:I\to\G_{q}$
by 
\[
c_{q}\lvert_{I_{i}}=g_{i}\quad \bmod q.
\]
\end{defn}
This quantity will appear again naturally in Section \ref{sec:Counting}
when we perform the lattice point count.
\subsection{Thermodynamics\label{sub:Thermodynamics}}

For a $T$-invariant probability measure $\mu$ on $\K$, let $h_{\mu}(T)$
denote the measure-theoretic entropy of $T$ with respect to $\mu$.
Let $\mathcal{M}(\K)^T$ denote the set of all $T$-invariant probability
measures on $\K$.

The \emph{pressure functional} is defined on $f\in L(K)$ by 
\[
\text{}P(f):=\sup_{\mu\in\mathcal{M}(\K)^{T}}\left(h_{\mu}(T)-\int_{\K}fd\mu\right).
\]
It follows from the variational principle that $P(-s\tau)$ is strictly
decreasing in a real parameter $s$ and has a unique positive zero
denoted by $s_{0}.$ Moreover it is known that in the current setting
$s_{0}=\delta$, where $\delta$ is the Hausdorff dimension of $K$. 

Let $L(\K)$ denote the Banach space of Lipschitz
functions on $\K$. For any real valued $f\in L(\K)$, the transfer operator $\L_{f}$
on $L(\K)$ is given by 
\begin{equation}
\L_{f}[G](x)=\sum_{Ty=x}e^{f(y)}G(y).\label{eq:transferdef}
\end{equation}
The basic spectral theory of transfer operators is given by
the Ruelle-Perron-Frobenius Theorem. We state this following Naud
\cite{NAUD}, the result can also be found in \cite{PPAST}. \\
\begin{thm}[Ruelle-Perron-Frobenius]\label{RPFTheorem}

\begin{enumerate}
\item There is a unique probability measure $\nu_{f}$ on $\K$ such that
$\L_{f}^{*}(\nu_{f})=e^{P(f)}\nu_{f}$. 
\item The maximal eigenvalue of $\L_{f}$ is $e^{P(f)}$ which belongs to
a unique positive eigenfunction $h_{f}\in L(\K)$ with $\nu_{f}(h_{f})=1$. 
\item The remainder of the spectrum of $\L_{f}$ is contained in a disc
of radius strictly less than $e^{P(f)}$. 
\end{enumerate}
\end{thm}

Our functional analysis takes place for the most part on the Banach
space $C^{1}(I)$ with the norm

\begin{equation}
\|f\|_{C^{1}(I)}=\|f\|_{\infty}+\|f'\|_{\infty},\label{eq:norms-1}
\end{equation}
or closely related spaces of vector valued functions. As in \cite{NAUD}
we need to note that Theorem \ref{RPFTheorem} extends reasonably
to $\L_{f}$ acting on $C^{1}(I)$ given $f\in C^{1}(I)$. In particular
$\L_{f}$ acting on $C^{1}(I)$ has the same spectral properties relative
to a positive eigenfunction $h_{f}\in C^{1}(I)$ such that $\L_{f}h_{f}=e^{P(f)}h_{f}$.
We also view $\nu_{f}$ as a measure on $I$ with support in
$\K$.

We will write simply $\L_{s}=\L_{-s\tau}$ in the sequel.

\section{Counting\label{sec:Counting}}


\subsection{From the lattice point count to the boundary dynamics}

We now show how one can adapt the work of Lalley \cite{LALLEYSYMB}
to get counting estimates in our setting.  Let $\G_{q}=\SL_{2}(\Z/q\Z).$ We convert questions about
the lattice point count in congruence classes into questions about
the $\R^{\G_{q}}$ valued function 
\[
N_{q}^{*}(a,\g_{0},\vp):=\sum_{\g\in\G\cup\{e\}\::d(o,\g\g_{0}o)-d(o,\g_{0}o)\leq a}G(\g\g_{0}o)\rho(\pi_{q}(\g)).\vp
\]

where

\begin{itemize}
\item $G$ is a non negative function on $\H\cup\R$ with the property that
there exist an integer $M$ and neighborhood $J_{M}$ of the length
$M$ cylinders in $I$ such that $G$ is locally constant on $J_{M}$.
We write $g$ for the restriction of $G$ to $\R$. 
\item $\vp\in\R^{\G_{q}}$ , $\pi_{q}:\G\to\G_{q}$ is reduction mod $q$
and $\rho$ is the right regular representation of $\G_{q}$. 
\item $o = i \in\H$ is our fixed origin and $\g_{0}\in\G\cup\mathrm{id}$. 
\end{itemize}

While this might seem mysterious, we explain as follows. 

\emph{Firstly, and most importantly, the Main Theorem \ref{thm:mainsimple}
stated in our Introduction is directly analogous to certain estimates
for $N_{q}^{*}(a,\mathrm{id,\varphi)}$ for suitable test $\vp$.}

\textbf{The distance $d$ vs the matrix norm $\|\g\|$. }One has the
identity 
\[
\|\g\|^{2}=2\cosh(d(i,\g i)).
\]
With this in hand and our choice $o=i$ of basepoint, the condition $d(i,\g\g_{0}i)-d(i,\g_{0}i)\leq a$
becomes 
\begin{equation*} 
\frac{\|\g\g_{0}\|}{\|\g_{0}\|}\leq R,
\end{equation*} 
where $R=\sqrt{2\cosh(a)}=e^{a/2}$. \footnote{ More precisely the condition $\frac{||\gamma \gamma_0||}{||\gamma||} \leq R$ corresponds to an inequality 
\begin{equation}d(i,\g\g_{0}i)-d(i,\g_{0}i) \leq 2 \log R + \log (1 + e^{-2d(i, \gamma_0i)}) + \log \left( 1 + \sqrt{ 1 - \frac{1}{R^4 \cosh^2 d(i, \gamma_0i)}}\right) = a +  \log (1 + e^{-2d(i, \gamma_0i)})  + O(e^{-2a}).  \label{footnoteeqn} \end{equation}
 The difference is only important insofar as it changes the leading constant in our main theorem.  }

\textbf{The parameter $\gamma_{0}$. }Our Main Theorem \ref{thm:mainsimple}
of the Introduction is obtained by setting $\g_{0}=\mathrm{id}.$
However, even to obtain this simplified version, consideration of
general $\g_{0}$ is necessary in order to set up the forthcoming
recursion over the tree-like $\G.$ This recursive formula leads to
the renewal equation.

\textbf{The function $G$. }This function allows one to perform sector
estimates by only counting lattice points that fall close to a prescribed
part of the boundary $\partial(\H)$ of hyperbolic space.

\textbf{Modular twisting. }Let us now explain the modular twisting
in the simple case that $G:=1.$ Recall that we are supposed to
be counting in a given congruence class $\xi\in\G_{q}.$ One can decompose
the characteristic function of the singleton set $\xi$ according
to its constant coefficient and a part orthogonal to constants, and
look at $N(a,\g_{0},\vp)$ with $\vp$ set in turn to these different
components. Since the estimate is additive one can estimate the corresponding
quantities separately. The key calculation is that

\[
N_q^*(a,\mathrm{id},\mathbf{1}_{\xi})=\sum_{\g\in\G\cup\{e\}\::d(o,\g o)\leq a}\rho(\pi_{q}(\g)).\mathbf{1}_{\xi}=\sum_{\g\in\G\cup\{e\}\::d(o,\g o)\leq a}\mathbf{1}_{\xi\pi_{q}(\gamma)}
\]
 so one obtains the congruence lattice point count from reading off
a coordinate of the vector valued $N_q^*(a,\mathrm{\mathrm{id,}\mathbf{1_{\xi}).}}$

\noindent \textbf{Remark.} Whenever we sum over semigroup elements we have the implied constraint that any concatenation in the summation condition be admissible; we will use  the notation $\sum^*$  to emphasize this.  For example, we will write 
\[ \sum^*_{\substack{\frac{\|\g\g_{0}\|}{\|\g_{0}\|}\leq R\\
\g\equiv\xi\bmod q
}} G(\g\g_{0}o) := 
\sum_{\substack{\frac{\|\g\g_{0}\|}{\|\g_{0}\|}\leq R\\
\g\equiv\xi\bmod q\\
\g \cdot \g_0 \mbox{ admissible} 
}}G(\g\g_{0}o).\]

The most general lattice point count that the upcoming estimates for
$N(a,\g_{0},\vp)$ will allow us to obtain is the following.

\begin{thm}[Main Theorem, elaborated]
 \label{thm:maintheoremelaborate}There exist $Q_{0}\in\mathbf{N}$,
$C>0$ and $\e>0$ such that for all $\g_{0}\in\G$, $\xi\in\SL_{2}(\Z/q\Z)$
and $q$ with $(Q_{0},q)=1$, 
\[
\sum^*_{\substack{\frac{\|\g\g_{0}\|}{\|\g_{0}\|}\leq R\\
\g\equiv\xi\bmod q
}
}G(\g\g_{0}o)=\frac{R^{2\delta_{\mathcal{}}}}{|\G_{q}|}\hat C_{*}(\gamma_{0},G\lvert_{\R})+O\left((\|G\|_{\infty}+\|[G\lvert_{\R}]'\|_{\infty})q^{C}R^{2(\delta_{\mathcal{}}-\e)}\right).
\]
Here $G$ is any function in $C^{1}(\H\cup\R)$ which is locally constant
on some neighborhood of the cylinders of length $M$ in $I$ for some
$M>0$. The constant $\hat C_{*}(\gamma_{0},G\lvert_{\R})>0$ is related to $C_*$ from 
\eqref{eq:C*} but modified in light of \eqref{footnoteeqn}. The implied constant depends on $M$. 
\end{thm}

We now show how to relate the quantities $N_{q}^{*}$ and the dynamics
on the boundary. As before, write $d_{E}$ for Euclidean distance
in the upper half plane. Let $\G^{(n)}$ denote those $\g\in\G$ which 
can be written as a reduced word in at least $n$ generators. If
$\gamma = g_{i_{1}}g_{i_{2}}\ldots g_{i_{n}}$ is written in reduced form
then we define
the \emph{shift} 
\[
\sigma:\G^{(n)}\to\G^{(n-1)},\quad\sigma(\gamma)=g_{i_{2}}\ldots g_{i_{n}}.
\]
We use the convention that  $\G^{(0)}=\G\cup\{e\}$ and $\sigma(g_{i})=e$
for all $1\leq i\leq k$. Throughout the rest of this section we always assume 
semigroup elements are written in their reduced form.

Define for $\g\in\G$ 
\[
\tau_{*}(\g)=d(o,\g o)-d(o,(\sigma\g)o).
\]
and define for $n\geq N$ and $\g\in\G^{(n)}$ 
\[
\tau_{*}^{N}(\g)=\sum_{j=0}^{N-1}\tau_{*}(\sigma^{j}\g)=d(o,\g o)-d(o,(\sigma^{N}\g)o).
\]
We can now recast $N_{q}^{*}$ as 
\[
N_{q}^{*}(a,\g_{0},\vp)=\sum_{n=0}^{\infty}\sum_{\g\in\G:\sigma^{n}\g=\g_{0}}G(\g o)\rho(\pi_{q}(\g\g_{0}^{-1}))\cdot\vp\mathbf{1}\{\tau_{*}^{n}(\g)\leq a\}.
\]
One obtains by this elementary argument a recursive formula called
the \emph{renewal equation}: 
\begin{equation}
N_{q}^{*}(a,\g_{0},\vp)=\sum_{\g:\sigma\g=\g_{0}}N_{q}^{*}(a-\tau_{*}(\g),\g,[\rho(\pi_{q}(\g\g_{0}^{-1}))\vp])+G(\g_{0} o)\vp\mathbf{1}\{a\geq0\}.\label{eq:finiterenewal}
\end{equation}

We will now `push to the boundary', replacing quantities with boundary
counterparts under the following Dictionary.

\begin{center}
\begin{tabular}{cc}
Inside $\H$ (lattice point count) & The boundary $\partial(\H)$\tabularnewline
\hline 
\label{table:dictionary}$\sigma$ & $T$\tabularnewline
$\tau_{*}$ & $\tau(x)$  (see Definition \ref{def:The-distortion-function-II})\tabularnewline
$\tau_{*}^{N}$ & $\tau^{N}(x)=\sum_{i=0}^{N-1}\tau(T^{i}x)$\tabularnewline
$G$ & $g=G\lvert_{I}$\tabularnewline
$\rho$ & The cocycle $c_{q}$ (see Definition \ref{def:modcocycle})\tabularnewline
$\rho(\pi_{q}(\g\g_{0}^{-1}))$ & $c_{q}^{N}(x):= c_{q}(T^{N-1}x)c_{q}(T^{N-2}x)\ldots c_{q}(Tx)c_{q}(x)$\tabularnewline
$N_{q}^{*}(a,\g_{0},\vp)$ & $N_{q}(a,x,\vp)=\sum_{n=0}^{\infty}\sum_{y:T^{n}y=x}g(y)\rho(c_{q}^{n}(y))\vp\mathbf{1}\{\tau^{n}(y)\leq a\}.$\tabularnewline
\end{tabular}
\par\end{center}

These new quantities play a central role in the remainder of the paper,
in place of their old counterparts.  We take this opportunity to outline the rest of this section. 

\begin{itemize} 
\item We would like to understand the quantity $N_{q}^{*}(a,\g_{0},\vp)$. It's not clear how to do this directly, so we compare it to $N_{q}(a,\g_{0}k_0,\vp)$. Unfortunately that comparison is only valid when $\g_0$ is a ``large'' group element (see Lemmas \ref{bridging} and   \ref{sandwich}), but we can arrrange that by repeated application of the finite renewal equation (see \eqref{eq:finiterenewal}) so we obtain Lemma \ref{lem:sandwich2}.  
\item Next we relate $N_{q}(a,\g k_0,\vp)$ to the transfer operators. This is done by means of the boundary renewal equation \eqref{eq:renewal} and a Laplace transform: we obtain \eqref{eq:nqdef}.
\item Spectral bounds for transfer operators (see Theorem \ref{thm:maintransfer}) together with equation \eqref{eq:nqdef} and the Laplace inversion formula give us good control on  $N_{q}(a,\g k_0,\vp)$: see Proposition \ref{boundaryprop}.
\item Finally we use control of  $N_{q}(a,\g k_0,\vp)$ to gain control of $N_{q}^{*}(a,\g_{0},\vp)$.
\end{itemize} 
We will now put this outline into practice.

Assume that $\g_0\neq 1$ (the case $\g_0=1$ follows from this consideration\footnote{By applying the renewal equation \eqref{eq:finiterenewal}, the quantity $N_q^*(a,1,\vp)$ is converted to a constant plus a finite sum of quantities of the form $N_q^*(\bullet,\gamma_0,\bullet)$ where $\g_0\neq 1$.}). In the Schottky setup, we say that
$k_0\in \tilde{I}_i$ is admissible for $\gamma_0$ if 
\begin{equation}\label{eq:gamma0}
 \g_0 = g_{i_1}\ldots g_{i_N}
\end{equation}
is reduced and  $i_N\neq i+k'\bmod 2k'$ . We fix such an admissible $k_0\in K$ now - if in the continued fractions setup this can be chosen arbitrarily in $K$.

\begin{lem}
\label{bridging}There is  $\k<1$ such that when 
\begin{equation}\label{eq:gammagamma0}
 \g = g_{j_0}\ldots g_{j_{n+N}}
\end{equation}
is a reduced word in $\Gamma$, and  $\g \g_0$ is also reduced/admissible then
\[
\tau_{*}^{n}(\g\g_{0})=\tau^{n}(\g\g_{0}k_{0})+O(\k^{N}).
\]
\end{lem}

\begin{proof}

Let $C$ and $\kappa$ be the constants from Lemma \ref{contracting}. Then

\begin{equation}
d_{E}(\g\g_{0}o,\g\g_{0}k_{0})\leq C \k^{n+N}\label{eq:convergence}.
\end{equation}
We also have 
\[
\tau_{*}(g_{j_{0}}g_{j_{1}}g_{j_{2}}\ldots g_{j_{n-1}}\ldots g_{j_{n+N}}\g_{0})=-\log|g_{j_{0}}'(g_{j_{1}}g_{j_{2}}\ldots g_{j_{n+N}}\g_{0}o)|+o(d_{E}(g_{j_{1}}g_{j_{2}}\ldots g_{j_{n+N}}\g_{0}o,\R)).
\]
the derivative here is for the action of $\G$ on the unit disc model obtained via $J$; a similar estimate is given in \cite[pg. 41]{LALLEYSYMB}. Note that the error term can be measured either in the unit disc model or the upper half plane model, as the two are bi-Lipschitz near $K$. It follows
then that 

\color{black}

\[
\tau_{*}(g_{j_{0}}g_{j_{1}}g_{j_{2}}\ldots g_{j_{n+N}}\g_{0})=-\log|g_{j_{0}}'(g_{j_{1}}g_{j_{2}}\ldots g_{j_{n+N}}\g_{0}o)|+O(\k^{n+N-1}).
\]
Since there is some uniform bound for the derivative of $\log|g_{i}'|$
close to the part of  $I$ where $g_i$ is an inverse branch of $T$, this together with \eqref{eq:convergence} implies
\[
\tau_{*}(g_{j_{0}}g_{j_{1}}g_{j_{2}}\ldots g_{j_{n+N}}\g_{0})=-\log|g_{j_{0}}'(g_{j_{1}}g_{j_{2}}\ldots g_{j_{n+N}}\g_{0}k_{0})|+O(\k^{n+N-1}).
\]
By iterating for $n$ steps and summing the geometric series it follows that 
\[
\tau_{*}^{n}(g_{j_{0}}g_{j_{1}}g_{j_{2}}\ldots g_{j_{n+N}}\g_{0})=-\log|(g_{j_{0}}g_{j_{1}}g_{j_{2}}\ldots g_{j_{n-1}})'(g_{j_n}\ldots g_{j_{n+N}}\g_{0}k_{0})|+O(\k^{N})
\]
or what is the same, 
\begin{equation}
\tau_{*}^{n}(\g\g_{0})=\tau^{n}(\g\g_{0}k_{0})+O(\k^{N}),\label{eq:bridging}
\end{equation}
 proving the Lemma.

 \end{proof}

\color{black}
\begin{lem}
\label{sandwich}Suppose $\vp$ is non
negative. There are $N_{0}$, $\k<1$ and $C$ depending on $G$ such
that if  
$N>N_{0}$ and 

\begin{equation*}
\g_1 = g_{r_1}\ldots g_{r_N} \g_0
\end{equation*}
is an admissible concatenation (hence $k_0$ is admissible for $\g_1$) then
\[
N_{q}(a-C\k^{N},\g_{1}k_{0},\vp)\leq N_{q}^{*}(a,\g_{1},\vp)\leq N_{q}(a+C\k^{N},\g_{1}k_{0},\vp).
\]
The inequalities are understood between functions on $\R^{\G_{q}}$. \end{lem}

\begin{proof}

We will use the fact that the map $\g\g_1\mapsto  \g\g_1  k_0$ on admissible concatenations
 intertwines the shift $\sigma$ and the map $T$. One has 
\begin{align}
N_{q}^{*}(a,\g_1,\vp)=\sum_{n=0}^{\infty} \sum^*_{\substack{ \g=g_{i_1}\ldots g_{i_n} }}
G(\g \g_1 o)\rho(\pi_{q}(\g))\vp\mathbf{1}\{\tau_{*}^{n}(\g\g_1)\leq a\}\label{eq:finiteN}
\end{align}
and 
\begin{align*}
N_{q}(a,\g_{1}k_{0},\vp)=\sum_{n=0}^{\infty} \sum^*_{\substack{ \g=g_{i_1}\ldots g_{i_n} }}g(\g\g_1k_0)\rho(\pi_{q}(\g))\vp\mathbf{1}\{\tau^{n}(\g\g_1 k_0)\leq a\}.
\end{align*}
These can now be compared term by term. If $N$ is large enough, depending
on $G$, then $G(\g \g_1 o)=g(\g\g_1k_0 )$ for all terms as all the $\g\g_1 o$ will
lie in the neighborhood $J_{M}$ where $G$ is locally constant. On the other hand, we have from
Lemma \ref{bridging} that 
\[
\mathbf{1}\{\tau_{*}^{n}(\g\g_1)\leq a\}\leq\mathbf{1}\{\tau^{n}(\g\g_1k_0)\leq a+C\k^{N}\}
\]
for some $C$. Also, in the other direction,
\[
\mathbf{1}\{\tau^{n}(\g\g_1k_0)\leq a-C\k^{N}\}\leq\mathbf{1}\{\tau_{*}^{n}(\g\g_1)\leq a\}.
\]
Given that $\vp$ and hence $\rho(\pi_{q}(\g))\vp$ are
positive functions, inserting these inequalities into \eqref{eq:finiteN}
gives the result after suitably choosing $N_0$.
\end{proof}

Following Lalley \cite[pg. 22]{LALLEYSYMB} we iterate the finite
renewal equation \eqref{eq:finiterenewal} to obtain 
\begin{align*}
N_{q}^{*}(a,\g_{0},\vp) & =\sum_{\g:\sigma^{n}\g=\g_{0}}N_{q}^{*}(a-\tau_{*}^{n}(\g),\g,\rho[\pi_{q}(\g\g_{0}^{-1})]\vp)\\
 & +\sum_{m=1}^{n-1}\sum_{\g:\sigma^{m}\g=\g_{0}}G(\g o)\rho[\pi_{q}(\g\g_{0}^{-1})]\vp\mathbf{1}\{a-\tau_{*}^{m}(\g)\geq0\}+G(\g_{0} o)\vp\mathbf{1}\{a\geq0\}.
\end{align*}
We want to increase $n$ so we note that the second line is bounded
by (recall $k$ is the number of intervals) 
\[
\sum_{m=0}^{n-1}k^{m}\|G\|_{\infty}\|\vp\|\ll\|G\|_{\infty}\|\vp\|k^{n}.
\]
We will eventually set 

\[
n=\lfloor ca\rfloor
\]
for small enough $c$. This gives 
\[
N_{q}^{*}(a,\g_{0},\vp)=\sum_{\g:\sigma^{n}\g=\g_{0}}N_{q}^{*}(a-\tau_{*}^{n}(\g),\g,\rho[\pi_{q}(\g\g_{0}^{-1})]\vp)+O(\|G\|_{\infty}\|\vp\|e^{(\log k)ca}).
\]
We can now use Lemma \ref{sandwich} to get:
\begin{lem}
\label{lem:sandwich2}Up to an error of $O(\|G\|_{\infty}\|\vp\|e^{(\log k)ca})$,
$N_{q}^{*}(a,\g_{0},\vp)$ is sandwiched between 
\[
\sum_{\g:\sigma^{n}\g=\g_{0}}N_{q}(a-\tau_{*}^{n}(\g)-C\k^{n},\g k_{0},\rho[\pi_{q}(\g\g_{0}^{-1})]\vp)
\]
and 
\[
\sum_{\g:\sigma^{n}\g=\g_{0}}N_{q}(a-\tau_{*}^{n}(\g)+C\k^{n},\g k_{0},\rho[\pi_{q}(\g\g_{0}^{-1})]\vp).
\]

\end{lem}
This sandwiching allows us to convert questions about $N_{q}^{*},$
and hence our main Theorem, to questions about $N_{q}.$ We leave
the relation for now since going any further in the comparison requires
results from later in the paper. Hopefully by now we have motivated
the study of $N_{q}$ and the dynamical system of Section \ref{sub:A-dynamical-system}.

\subsection{The renewal equation: boundary version}

The quantity $N_{q}$ also satisfies a version of the renewal equation:
we first describe a simple version without any congruence aspect.
Let $g\in C^{1}(I)$ as before.

We define 
\[
N(a,x)=\sum_{n=0}^{\infty}\sum_{y:T^{n}y=x}g(y)\mathbf{1}\{\tau^{n}(y)\leq a\},
\]
where $\mathbf{1}\{\tau^{n}(y)\leq a\}$ is the characteristic function
of $\{\tau^{n}(y)\leq a\}$. Only finitely many of the $n$ give a
contribution to the sum, since $\tau$ is eventually positive. The
\textit{renewal equation} states 
\begin{equation}
N(a,x)=\sum_{y:Ty=x}N(a-\tau(y),y)+g(x)\mathbf{1}\{a\geq0\}.\label{eq:renewal}
\end{equation}
This is related to the transfer operator $\L_{-s\tau}$ by taking
a Laplace transform in the $a$ variable. If one defines 
\[
n(s,x)=\int_{-\infty}^{\infty}e^{-sa}N(a,x)da
\]
then \eqref{eq:renewal} is transformed into 
\[
n(s,x)=[\L_{s}n(s,\cdot)](x)+\frac{g(x)}{s},
\]
where
$\L_{s}[f]$
is the transfer operator defined in Section \ref{sub:Thermodynamics}.
The former equation can be recast to

\[
s.n(s,\bullet)=(1-\L_{s})^{-1}g
\]

We now adapt our formulae to take account of the congruence aspect.
The congruence version of the renewal equation at level $q$ concerns
the quantity 
\[
N_{q}(a,x,\vp):=\sum_{n=0}^{\infty}\sum_{y:T^{n}y=x}g(y)\rho(c_{q}^{n}(y))\vp\mathbf{1}\{\tau^{n}(y)\leq a\}\in\C^{\G_{q}}
\]
from before. This congruence renewal equation reads 
\[
N_{q}(a,x,\vp)=\sum_{y:Ty=x}\rho(c_{q}(y))N_{q}(a-\tau(y),y,\vp)+g(x)\vp\mathbf{1}\{0\leq a\}.
\]
Consider  the \emph{congruence transfer operator} $\L_{s,q}$ on  \emph{$\C^{\Gamma_{q}}$-}valued\emph{ }functions
 defined as follows:
\begin{equation}
\L_{{s},q}[F](x):=\sum_{Ty=x}e^{-s\tau(y)}c_{q}(y).F(y)\label{eq:skewtransfer}.
\end{equation}
where $c_q$ is the modular cocycle given in 
Definition \ref{def:modcocycle}. Then parallel arguments to before give for 
\[
n_{q}(s,x,\vp)=\int_{-\infty}^{\infty}e^{-sa}N_{q}(a,x,\vp)da
\]
the formula
\begin{equation}
sn_{q}(s,x,\vp)=[(1-\L_{s,q})^{-1}g\otimes\vp](x)\label{eq:nqdef}
\end{equation}
where $g\otimes\vp$ is the vector valued function taking $x\mapsto g(x)\vp$.

 \subsection{Spectral theory of transfer operators}

Recall that we work with the Banach space $C^{1}(I)$ with norm as
in \eqref{eq:norms-1} and the similar Banach spaces $C^{1}(I;\C^{\G_{q}})$
of $\C^{\G_{q}}$-valued functions. In Theorem  \ref{thm:maintransfer} we summarized
the spectral properties of $\L_{s,q}$ that we prove in this paper,
and that will be used to estimate equation \eqref{eq:nqdef}.
The proof of Theorem \ref{thm:maintransfer} is deferred to Sections
\ref{largeimaginarypart} and \ref{smallimaginarypart}. We now continue
with our counting estimates using Theorem \ref{thm:maintransfer}
as a given.

\subsection{Continuing the count}

Notice that $N_{q}$ and hence $n_{q}$ are linear in $\vp$. We split
into two cases as we can write 
\[
\vp=\vp_{0}+\vp'
\]
where $\vp_{0}$ is constant and $\vp'$ is orthogonal to constants.
The analysis of $N_{q}(a,x,\vp_{0})$ boils down to that of $N(a,x)$,
which is in principle understood without any of the results of this
paper. We take up the analysis in the case that 
\[
\vp'\in\C^{\G_{q}}\ominus1,
\]
that is, orthogonal to constants. Assume this is the case from now
on.

One obtains from \eqref{eq:nqdef} and Theorem \ref{thm:maintransfer}
that for any $\eta>0$ 
\begin{equation}
|s|\|n_{q}(s,\bullet,\vp')\|_{C^{1}}\leq\begin{cases}
Cq^{C}(1-\rho_{0})^{-1}\|g\otimes\vp\|_{C^{1}}\text{ if }|b|\leq b_{0}\\
C_{\eta}|b|^{1+\eta}(1-\rho_{\eta})^{-1}\|g\otimes\vp\|_{C^{1}}\text{ if }|b|>b_{0}
\end{cases}\label{eq:nqbound}
\end{equation}
with the same quantifiers and constants as in Theorem \ref{thm:maintransfer}.
Consolidating constants, for any $\eta>0$ there is $C'=C'(\eta)$
such that 
\begin{equation}
|s|\|n_{q}(s,\bullet,\vp')\|_{C^{1}}\leq C'\max(q^{C},|b|^{1+\eta})\|g\otimes\vp\|_{C^{1}}\label{eq:nqdecay}
\end{equation}
whenever $|a-\delta|<\e$ for some sufficiently small $\e$.

We also note that given the bounds in Theorem \ref{thm:maintransfer},
it follows that the correspondence 
\[
s\mapsto(1-\L_{s,q})^{-1}g\otimes\vp'
\]
gives a holomorphic family of $C^{1}$ functions in the region $|a-\delta|<\e$
for fixed $g$ and $\vp'$, hence $n_{q}(s,x,\vp')$ is holomorphic
for $s$ in this region. This is essential for the contour shifting
argument to follow. Now we follow technical work of Bourgain, Gamburd
and Sarnak \cite[pp. 25-26]{BGS2} to extract information about
$N_{q}(a,x,\vp')$.

 Following \cite[equation (9.4)]{BGS2},  let $k$ be a smooth nonnegative function on $\R$ such that $
\int k=1$,  $\mathrm{support}(k)\subset\left[1,1\right]$
and\footnote{The assumption that $\hat{k}$ has stretched exponential decay is overly strong here: it would be sufficient for example that $\hat{k}$ be uniformly bounded and in $L^1$ of any vertical line in $\C$ with real part sufficiently close to $\delta$.}
\[
|\hat{k}(\xi)|\leq B\exp(-|\xi|^{1/2})
\text{ 
for some $B$,}\] where 
\[
\hat{k}(\xi):=\int_{\R}e^{-\xi t}k(t)dt.
\]
Then let for small $\lambda>0$ 
\[
k_{\lambda}(t)=\lambda^{-1}k(t\lambda^{-1}),
\]
this has the effect that 
\begin{equation}
\hat{k_{\lambda}}(\xi)=\hat{k}(\lambda\xi),\quad|\hat{k_{\lambda}}(\xi)|\leq B\exp(-|\lambda\xi|^{1/2}).\label{eq:scaledfourier}
\end{equation}
Consider the smoothed quantity of interest 
\[
\int_{-\infty}^{\infty}k_{\lambda}(t)N_{q}(a+t,x,\vp')dt=\frac{1}{2\pi i}\int_{s\in\delta+i\R}e^{as}n_{q}(s,x,\vp')\hat{k_{\lambda}}(s)ds.
\]
by inverting the Laplace transform and interchanging the order of
integration. From \eqref{eq:nqdecay}, $n_{q}$ is well enough behaved
that this is possible. For technical reasons let $\e'=\min(\delta/2,\e/2)$.
We can shift the contour to $\Re(s)=\delta-\e'$ to get that the above
is the same as 
\begin{align*}
 & \frac{1}{2\pi i}\int_{s\in\delta-\e'+i\R}e^{as}n_{q}(s,x,\vp')\hat{k_{\lambda}}(s)ds\\
 & =\frac{1}{2\pi}e^{a(\delta-\e')}\int_{\theta\in\R}e^{ai\theta}n_{q}(\delta-\e'+i\theta,x,\vp')\hat{k_{\lambda}}(\delta-\e'+i\theta)d\theta
\end{align*}
where $s=\delta-\e'+i\theta$. Putting in the bound \eqref{eq:nqbound}
for $n_{q}$ together with \eqref{eq:scaledfourier} gives the new
bound 
\begin{align*}
 & \frac{BC'}{2\pi}e^{a(\delta-\e')}\|g\otimes\vp'\|_{C^{1}}\left(q^{C}\int_{|\theta|\leq b_{0}}|\delta-\e'+i\theta|^{-1}e^{-|\lambda(\delta-\e'+i\theta)|^{1/2}}d\theta\right.\\
 & +\left.\int_{|\theta|>b_{0}}|\delta-\e'+i\theta|^{-1}|\theta|^{1+\eta}e^{-|\lambda(\delta-\e'+i\theta)|^{1/2}}d\theta\right)\\
 & \leq\frac{BC'}{2\pi}e^{a(\delta-\e')}\|g\otimes\vp'\|_{C^{1}}\left(\frac{4q^{C}b_{0}}{\delta-\e'}+C''|\lambda|^{-1-\eta}\right)
\end{align*}
for some new absolute constants $C',C''$. Putting this together (choosing
$\eta=1$ is enough) gives
\begin{lem}
\label{congtermrenewal}There is $Q_{0}>0$ provided by Theorem \ref{thm:maintransfer}
and positive constants $\e'$, $C$, $\k_{1}$, $\k_{2}$ such that
for $q$ with $(Q_{0},q)=1$ and any $g\in C^{1}(I)$, $\vp'\in\C^{\G_{q}}\ominus1$
we have 
\[
\|\int_{-\lambda}^{\lambda}k_{\lambda}(t)N_{q}(a+t,x,\vp')dt\|<e^{a(\delta-\e')}\|g\otimes\vp'\|_{C^{1}}\left(\k_{1}q^{C}+\k_{2}|\lambda|^{-2}\right)
\]
where the norm on the left hand side is the one in $\C^{\G_{q}}$. 
\end{lem}
We now describe $N_{q}(a,x,\vp_{0})$ with $\vp_{0}$ a constant function.
In this case the counting reduces to the non congruence setting. The
following is a straightforward adaptation of \cite[Proposition 10.2]{BGS2}
to our setting. This effectivizes work of Lalley \cite{LALLEYSYMB},
using the work of Naud \cite{NAUD} as input to get a power saving error term. Let $\underline{1}$
be the constant function in $\C^{\G_{q}}$ taking on the value 1.
\begin{lem}
\label{maintermrenewal}There exists $\e''>0$ such that for any $q$,
$g\in C^{1}(I)$ we have 
\[
\int_{-\lambda}^{\lambda}k_{\lambda}(t)N_{q}(a+t,x,\underline{1})dt=C(x,g)e^{\delta a}\underline{1}+O(\|g\|_{C^{1}}|\G_{q}|\lambda^{-3}e^{(\delta-\e'')a}),
\]
where 
\[
C(x,g)=\left(\frac{\int gd\nu_{-\delta\tau}}{\delta\int\tau d\nu_{0}}\right)h_{-\delta\tau}(x).
\]
is a $C^{1}$ function of $x$ and the error is estimated in $C^{1}$
norm, and $\nu,h$ are the measures and functions we defined in Theorem
\ref{RPFTheorem}. 
\end{lem}
We remark that the $|\G_{q}|\|g\|_{C^{1}}$ in the error term above
comes from $\|g\otimes\vp_{0}\|_{C^{1}}$. We can now put these Lemmas
together to get
\begin{prop}
\label{boundaryprop}There exists $Q_{0}>0$ provided by Theorem \ref{thm:maintransfer}
such that when $(Q_{0},q)=1$, the following holds. There is $\e>0$
such that for any non negative $\vp\in\R^{\G_{q}}\subset\C^{\G_{q}}$,
\[
N_{q}(a,x,\vp)=\frac{C(x,g)e^{\delta a}\langle\vp,\underline{1}\rangle\underline{1}}{|\G_{q}|}+O\left(e^{(\delta-\e)a}q^{C}\|g\|_{C^{1}}\|\vp\|\right)
\]
where $\langle\cdot,\cdot\rangle$ is the standard inner product. \end{prop}
\begin{proof}
Decompose $\vp$ as 
\[
\vp=\frac{\langle\vp,\underline{1}\rangle\underline{1}}{|\G_{q}|}+\vp'.
\]
Then Lemmas \ref{congtermrenewal} and \ref{maintermrenewal} give
that 
\begin{align*}
\int_{-\lambda}^{\lambda}k_{\lambda}(t)N_{q}(a+t,x,\vp)dt & =\frac{C(x,g)e^{\delta a}\langle\vp,\underline{1}\rangle\underline{1}}{|\G_{q}|}\\
 & +e^{a(\delta-\e)}O\left(\|g\|_{C^{1}}\|\vp\|(\k_{1}q^{C}+\k_{2}\lambda^{-2}+\lambda^{-3})\right)
\end{align*}
by using that 
\[
\|g\otimes\vp'\|_{C^{1}}\leq\|\vp'\|\|g\|_{C^{1}}
\]
and replacing $\e',\e''$ with a new small enough $\e$. Now taking
$\lambda=e^{-a\e/6}$ we have that the error term is 
\[
e^{a(\delta-\e/2)}O(q^{C}\|g\|_{C^{1}}\|\vp\|).
\]
Since $\vp$ is non negative, $N_{q}(a,x,\vp)$ is increasing in $a$
and hence 
\[
N_{q}(a-\lambda,x,\vp)\leq\int_{-\lambda}^{\lambda}k_{\lambda}(t)N_{q}(a+t,x,\vp)dt\leq N_{q}(a+\lambda,x,\vp)
\]
which is enough to get the result given the exponentially shrinking
$\lambda$, by replacing $\e$ with some smaller value. 
\end{proof}
With the precise asymptotics of Proposition \ref{boundaryprop} at
hand, we return to estimating $N_{q}^{*}(a,\g_{0},\vp)$. Using Lemma
\ref{lem:sandwich2} along with Proposition \ref{boundaryprop} gives

\begin{align*}
N_{q}^{*}(a,\g_{0},\vp) & =\left(1+O(\delta C\k^{n})\right)\frac{e^{\delta a}}{|\G_{q}|}\langle\vp,\underline{1}\rangle\underline{1}\sum_{\g:\sigma^{n}\g=\g_{0}}C(\g k_{0},g)e^{-\delta\tau_{*}^{n}(\g)}\\
 & +O\left(q^{C}\|g\|_{C^{1}}\|\vp\|e^{(\delta-\e)a}\sum_{\g:\sigma^{n}\g=\g_{0}}e^{-(\delta-\e)\tau_{*}^{n}(\g)}\right)+O(\|G\|_{\infty}\|\vp\|e^{(\log k)ca}).
\end{align*}
Given that $n=\lfloor ca\rfloor$ for some small $c$ yet to be chosen,
the $\k^{n}$ term will not be significant. We do however have to
describe the terms 
\[
\sum_{\g:\sigma^{n}\g=\g_{0}}C(\g k_{0},g)e^{-\delta\tau_{*}^{n}(\g)}
\]
and 
\[
\sum_{\g:\sigma^{n}\g=\g_{0}}e^{-(\delta-\e)\tau_{*}^{n}(\g)}.
\]
The latter can be bounded using Lemma \ref{bridging} with $N=0$
to give $\tau_{*}^{n}(\g)=\tau^{n}(\g k_{0})+O(1)$ and hence 
\begin{equation}
\sum_{\g:\sigma^{n}\g=\g_{0}}e^{-(\delta-\e)\tau_{*}^{n}(\g)}\ll\sum_{k:T^{n}k=\g_{0}k_{0}}e^{-(\delta-\e)\tau^{n}(k)}=[\L_{-(\delta-\e)}^{n}1](\g_{0}k_{0}).\label{eq:taubridged}
\end{equation}
We know that $\L_{-(\delta-\e)\tau}$ is bounded by $\exp(P(-(\delta-\e)\tau))$
by the Ruelle-Perron-Frobenius theorem. We now therefore require $n<\frac{a\e}{2P(-(\delta-\e)\tau)}$
so that 
\[
[\L_{-(\delta-\e)}^{n}1](\g_{0}k_{0})\ll\exp(nP(-(\delta-\e)\tau))\ll\exp(a\e/2).
\]

To describe the main term 
\begin{equation}
\frac{e^{\delta a}}{|\G_{q}|}\langle\vp,\underline{1}\rangle\underline{1}\sum_{\g:\sigma^{n}\g=\g_{0}}C(\g k_{0},g)e^{-\delta\tau_{*}^{n}(\g)},\label{eq:mainterm}
\end{equation}
we require the following result of Lalley (cf. \cite[Theorem 4]{LALLEYSYMB}).
It says that there is a version of the maximal eigenfunction $h_{-\delta\tau}$
on $\G$, as opposed to $K$.
\begin{lem}
\label{hinside}Fix $k_{0}\in K$. There exist a unique positive function
$h_{*}:\G\to\R$ and $\theta>1$ so that if $\g\in\G^{(n)}$ 
\[
h_{*}(\g)=h_{-\delta\tau}(\g k_{0})+O(\theta^{-n}).
\]
Also, for all $\g\in\G$, 
\begin{equation}
h_{*}(\g)=\sum_{\g':\sigma(\g')=\g}e^{-\delta\tau_{*}(\g')}h_{*}(\g').\label{eq:stationaryhstar}
\end{equation}

\end{lem}
Now recall the definition of $C(\cdot,g)$ from Lemma \ref{maintermrenewal}.
If we define the corresponding function on $\G$ according to the
pairing of $h_{*}$ with $h_{-\delta\tau}$, 
\begin{equation}
C_{*}(\g,g)=\left(\frac{\int gd\nu_{-\delta\tau}}{\delta\int\tau d\nu_{0}}\right)h_{*}(\g),\label{eq:C*}
\end{equation}
we get from Lemma \ref{hinside} that 
\[
C_{*}(\g,g)=C(\g k_{0},g)+O(\|g\|_{C^{1}}\theta^{-n})
\]
when $\g\in\G^{(n)}$. This means that the main term contribution
\eqref{eq:mainterm} to $N_{q}^{*}(a,\g_{0},\vp)$ is 
\begin{align*}
 & \frac{e^{\delta a}}{|\G_{q}|}\langle\vp,\underline{1}\rangle\underline{1}\left(\sum_{\g:\sigma^{n}\g=\g_{0}}C_{*}(\g,g)e^{-\delta\tau_{*}^{n}(\g)}+O(\|g\|_{C^{1}}\theta^{-n}\sum_{\g:\sigma^{n}\g=\g_{0}}e^{-\delta\tau_{*}^{n}(\g)})\right)\\
 & =\frac{e^{\delta a}}{|\G_{q}|}C_{*}(\g_{0},g)\langle\vp,\underline{1}\rangle\underline{1}+e^{\delta a}O(\theta^{-n}\|\vp\|\|g\|_{C^{1}})
\end{align*}
by using \eqref{eq:stationaryhstar} and a calculation similar to
that in \eqref{eq:taubridged} to give 
\[
\sum_{\g:\sigma^{n}\g=\g_{0}}e^{-\delta\tau_{*}^{n}(\g)}\ll[\L_{-\delta}^{n}1](\g_{0}k_{0})\ll 1 .
\]

We now let $n=\lfloor ca\rfloor$ with 
\[
c=\min\left(\frac{\delta-\e}{4\log k},\frac{\e}{2P(-(\delta-\e)\tau)}\right).
\]
Then the result of the preceding discussion is that 
\[
N_{q}^{*}(a,\g_{0},\vp)=\frac{e^{\delta a}}{|\G_{q}|}C_{*}(\gamma_{0},g)\langle\vp,\underline{1}\rangle\underline{1}+O\left((\|\vp\|(\|g\|_{C^{1}}+\|G\|_{\infty})q^{C}e^{(\delta-\e')a}\right)
\]
for some $\e'=\e'(\k,\theta,\e,\A)$ . When $\vp(\g)=\mathbf{1}\{\g=\xi\}$
we have that 
\[
\langle\vp,\underline{1}\rangle=1
\]
and hence evaluating $N_{q}^{*}(a,\g_{0},\mathbf{1}\{\g=\xi\})$
gives 
\[
\sum^*_{\substack{\g\in\G \::d(o,\g\g_{0}o)-d(o,\g_{0}o)\leq a\\
\pi_{q}(\g)=\xi
}
}G(\g\g_{0}o)=\frac{e^{\delta a}}{|\G_{q}|}C_{*}(\gamma_{0},g)+O\left((\|g\|_{C^{1}}+\|G\|_{\infty})q^{C}e^{(\delta-\e')a}\right).
\]
 This proves our Main Theorem \ref{thm:maintheoremelaborate}, given
Theorem \ref{thm:maintransfer}.

\section{Bounds for transfer operators: large imaginary part}

\label{largeimaginarypart} 

In this section we will prove Part \ref{enu:mainlargeimaginary} of
Theorem \ref{thm:maintransfer}.


\subsection{Non local integrability}

Recall from  Section \ref{continuedfractions}
the set $I$, $K$, the map $T:I\to\R$, the cocycles $c_{q}$ and
$\G$. We need to introduce symbolic dynamics.  
We write $A$ for the $k\times k$ matrix
with $(i,j)$ entry equal to $1$ if $T(I_{i})\supset I_{j}$ and $0$ otherwise.
Such a matrix $A$ is called
the \emph{transition matrix. } We say that a sequence $(i_{j})$ with entries in
$1,\ldots,k$ is admissible if $T(i_{j})\supset  i_{j+1}$ for all $j$
in the index set of the sequence. When $T(I_i)\supset I_j$ we define $T_{i}^{-1}$
on $I_{j}$ to be the unique locally defined branch of $T^{-1}$ that
maps $I_{j}$ to $I_{i}.$

Let $\Sigma_{A}^{+}$ (resp. $\Sigma_{A}^{-})$ be the space of positively
(resp. negatively) indexed admissible sequences on $\{1,\ldots,k\}.$ We define
for $\xi\in\Sigma_{A}^{-}$ the function 
\begin{equation}
\Delta_{\xi}(u,v)=\sum_{i=0}^{\infty}\tau(T_{\xi_{-i}}^{-1}\circ\ldots\circ T_{\xi_{0}}^{-1}u)-\tau(T_{\xi_{-i}}^{-1}\circ\ldots\circ T_{\xi_{0}}^{-1}v)\label{eq:Delta}
\end{equation}
on $I_{j}\times I_{j}$ such that $T(I_{\xi_{0}})\supset I_{j}$.
It follows from the expanding property of $T$ that $\Delta_{\xi}$
is $C^{1}$ where it is defined. Naud (following others) defines a
temporal distance function 
\begin{equation}
\varphi_{\xi,\eta}(u,v)=\Delta_{\xi}(u,v)-\Delta_{\eta}(u,v)\label{eq:temporaldistance}
\end{equation}
which is defined for each $\xi,\eta\in\Sigma_{A}^{-}$ and $u,v\in I_{j}$

\begin{defn}[Non local integrability (NLI)]
\label{NLI}An eventually positive function $\tau$ has property
(NLI) if there are $j_{0}\in\{1,\ldots,k\}$, $\xi,\eta\in\Sigma_{A}^{-}$
with $T(I_{\xi_{0}})\cap T(I_{\eta_{0}})\supset I_{j_{0}}$ and $u_{0},v_{0}\in K\cap I_{j_{0}}$
such that 
\[
\frac{\partial\varphi_{\xi,\eta}}{\partial u}(u_{0},v_{0})\neq0.
\]

\end{defn}



\begin{prop}\label{distortionprop}
 The distortion functions $\tau$ and $\hat \tau$ have the non local integrability property.
\end{prop}

\begin{proof} 
In the two cases of Schottky semigroups and the continued fractions semigroups we are considering, we always have two hyperbolic 
elements $h_i:=g_i^{-1}, h_j:= g_j^{-1}$
(with $g_i, g_j$ from the generating set) satisfying (1): $T|_{I_i}=h_i$ and $T|_{I_j}=h_j$,
(2) the $h_i$ and $h_j$ have distinct repelling (resp. attractive) fixed points on $\br\cup\{\infty\}$
and (3) the semigroup generated by $h_i$ and $ h_j$ consists of hyperbolic elements. 
 Given such elements, Naud's argument  in \cite[Proof of Lemma 4.4]{NAUD} 
shows the non local integrability properties of $\hat \tau(x)=\log |T'(x)|$ and $\tau(x)$ . \end{proof}

\subsection{Beginning Dolgopyat's argument}

One novelty of this paper is the following version of \cite[Theorem 2.3]{NAUD}
that is uniform in the congruence aspect. 
\begin{prop}
\label{bigimaginaryprop}There is $b_{0}>0$ such that part \ref{enu:mainlargeimaginary}
of Theorem \ref{thm:maintransfer} holds. That is, for any $\eta>0$,
there is $0<\rho_{\eta}<1$ such that 
\[
\|\L_{s,q}^{m}\|_{C^{1}}\ll_{\eta}|b|^{1+\eta}\rho_{\eta}^{m}
\]
when $|b|>b_{0}$ and $q\in\mathbf{N}$, as in Theorem \ref{thm:maintransfer}. 
\end{prop}
We now show how to relate this Proposition to the construction of
certain Dolgopyat operators. Recall the Ruelle-Perron-Frobenius Theorem (Theorem \ref{RPFTheorem})
and its notation. Let $h_{a}$ be the normalized positive eigenfunction
of $\L_{-a\tau}$ corresponding to the maximal eigenvalue $\exp(P(-a\tau))$.
We set 
\[
\tau_{a}=-a\tau-P(-a\tau)-\log(h_{a}\circ T)+\log(h_{a}).
\]
We now renormalize our transfer operators by defining 
\[
L_{s,q}:=\L_{\tau_{a}-ib\tau,q}.
\]
This is the same as 
\begin{equation}
L_{s,q}=\exp(-P(-a\tau))M_{h_{a}}^{-1}\L_{s,q}M_{h_{a}}\label{eq:perturbed}
\end{equation}
where $M_{h_{a}}$ is multiplication by $h_{a}$. It now follows by
arguments as in Naud \cite[pg. 132]{NAUD} that it is enough to prove
Proposition \ref{bigimaginaryprop} and Theorem \ref{thm:maintransfer}
with $L_{s,q}$ in place of $\L_{s,q}$. We also note here that the
maximal eigenfunction of $L_{a}$ is the constant function, with eigenvalue
$1$, that is $L_{a}1=1$ for $a\in\R$.

The rest of the passage to the estimates in the next section is routine
but we give some of the details for completeness. One shows that in
order to prove Proposition \ref{bigimaginaryprop} it is enough to
prove 
\begin{lem}
\label{powers} With the same conditions as Theorem \ref{thm:maintransfer},
there are $N>0$ and $\rho\in(0,1)$ such that when $|a-\delta|$
is sufficiently small and $|b|$ is sufficiently large we have 
\[
\int_{K}|L_{s,q}^{nN}W|^{2}d\nu_{0}\leq\rho^{n},
\]
where $W\in C^{1}(I;\C^{\G_{q}})$, $d\nu_{0}=h_{-\delta\tau}\nu_{-\delta\tau}$
is the Gibbs measure on $K$, and $\|W\|_{(b)}\leq1$, which stands
for the warped Sobolev norm 
\[
\|W\|_{(b)}=\|W\|_{\infty}+|b|^{-1}\|W'\|_{\infty}.
\]
These estimates are uniform in $q$. 
\end{lem}
This corresponds to \cite[Theorem 3.1]{OW} in the work of Oh and
Winter and is the uniform version of \cite[Proposition 5.3]{NAUD}.

Lemma \ref{powers} implies Proposition \ref{bigimaginaryprop} by
the use of a priori estimates for the transfer operators that allow
one to convert an $L^{2}$ estimate into a $C^{1}$ bound. These estimates
are given in \cite[Lemma 5.2]{NAUD} for complex valued functions.
They are however easily proved for vector valued functions giving
\begin{lem}
\label{apriori} There are $\k_{1},\k_{2},a_{0},b_{0}>0$ and $R<1$
such that for $|a-\delta|<a_{0}$ and $|b|>b_{0}$ we have for all
$f\in C^{1}(I;\C^{\G_{q}})$ 
\begin{equation}
\|[L_{s,q}^{n}f]'\|_{\infty}\leq\k_{1}|b|\|L_{a}^{n}f\|_{\infty}+R^{n}\|L_{a}^{n}|f'|\|_{\infty},\label{eq:apriori1}
\end{equation}
and 
\begin{equation}
\|L_{\delta,q}^{n}f\|_{\infty}\leq\int_{K}|f|d\nu_{0}+\k_{2}R^{n}\|f\|_{L(K)}.\label{eq:apriori2}
\end{equation}

\end{lem}
Lemma \ref{apriori} together with Lemma \ref{powers} imply Proposition
\ref{bigimaginaryprop} by arguments appearing in \cite[pp. 133-134]{NAUD}.
Roughly speaking the ingredients are Cauchy-Schwarz to access Lemma
\ref{powers}, remarks regarding the behaviour of $\tau_{a}^{m}$
for $a$ close to $\delta$ that appear elsewhere in this paper, and
splitting up exponents in the form $m=nN+r$.

The proof of Lemma \ref{powers} proceeds through the construction
of certain Dolgopyat operators that we give in the next section. 

\subsection{Construction of uniform Dolgopyat operators}

\label{operators}We follow the notation of Naud \cite{NAUD}. For
$A>0$ we consider the cone 
\[
\CC_{A}:=\{H\in C^{1}(I):H>0\text{ and }|H'(x)|\leq AH(x)\:\text{ for all }x\in I\:\}.
\]
In this section we establish a uniform version of the key Lemma of
Naud \cite[Lemma 5.4]{NAUD}. This is also analogous to \cite[Theorem 3.3]{OW}. 
\begin{lem}[Construction of uniform Dolgopyat operators]
\label{core}Suppose $\tau$ has the (NLI) property. There exists
$N>0,$ $A>1$ and $\rho\in(0,1)$ such that for all $s=a+ib$ with
$|a-\delta|$ small and $|b|>b_{0}$ large, there exists a finite
set of operators $(\N_{s}^{J})_{J\in\E_{s}}$ that are bounded on
$C^{1}(I)$ and satisfy the following three conditions 
\begin{enumerate}
\item \label{corefirst} The cone $\CC_{A|b|}$ is stable by $\N_{s}^{J}$
for all $J\in\E_{s}$. 
\item \label{coresecond} For all $H\in\CC_{A|b|}$ and all $J\in\E_{s}$,
\[
\int_{K}|\N_{s}^{J}H|^{2}d\nu_{0}\leq\rho\int_{K}|H|^{2}d\nu_{0}.
\]

\item \label{corethird} Given $H\in\CC_{A|b|}$ and $f\in C^{1}(I;\C^{\G_{q}})$
such that $|f|\leq H$ and $|f'|\leq A|b|H$, there is $J\in\E_{s}$
with 
\[
|L_{s,q}^{N}f|\leq\N_{s}^{J}H,\quad\text{and}\:\:|(L_{s,q}^{N}f)'|\leq A|b|\N_{s}^{J}H.
\]

\end{enumerate}
\end{lem}
When we write $|f|$ for $f\in C^{1}(I;\C^{\G_{q}})$ we refer to
the function obtained by taking pointwise Euclidean ($l^{2}$) norms.
We now show that the existence of these operators implies Lemma \ref{powers}.
\begin{proof}[Proof that Lemma \ref{core} implies Lemma \ref{powers}]

Given this construction (Lemma \ref{core}), Lemma \ref{powers} is
proved following the argument of \cite[pg. 21]{OW} or one in \cite[pg. 135]{NAUD}.
Indeed given non zero $f\in C^{1}(I;\C^{\G_{q}})$ with $\|f\|_{(b)}\leq1$
(cf. Lemma \ref{powers} for the definition of $\|\|_{(b)}$), we
define 
\[
H=\|f\|_{(b)}1.
\]
One sees that $H$ and $f$ are as in Lemma \ref{core}, that is,
$H\in\CC_{A|b|}$, $|f|\leq H$, and $|f'|\leq A|b|H$ as $A>1$.
One gets then by part \ref{corethird} of Lemma \ref{core} that 
\[
|L_{s,q}^{N}f|\leq\N_{s}^{J}H,\quad|(L_{s,q}^{N}f)'|\leq A|b|\N_{s}^{J}H
\]
for some $J\in\E_{s}$. Since $\CC_{A|b|}$ is stable under the $N_{s}^{J}$
one can repeat this to get for some sequence $J_{1},\ldots,J_{n}\in\E_{s}$
that 
\[
\int_{K}|L_{s}^{nN}f|^{2}d\nu_{0}\leq\int_{K}|\N_{s}^{J_{N}}\ldots\N_{s}^{J_{1}}H|^{2}d\nu_{0}\leq\rho^{n}\int_{K}|H|^{2}d\nu_{0}\leq\rho^{n}
\]
by using part \ref{coresecond} of Lemma \ref{core}. 
\end{proof}
The first two properties of Lemma \ref{core} were proved by Naud
in \cite{NAUD}; we follow closely Naud's construction of the operators
in the following.

\subsection{Consequences of non local integrability (NLI)}

Naud notes the following consequence of (NLI) that we will use later. 
\begin{lem}[Proposition 5.5 of \cite{NAUD}]
\label{NLIcons}If $\tau$ has property (NLI), there are $m,m',N_{0}>0$
such that for all $N>N_{0}$, there are two branches $\a_{1}^{N},\a_{2}^{N}$
of $T^{-N}$ with 
\[
m'\geq\left|\frac{d}{du}[\tau^{N}\circ\a_{1}^{N}-\tau^{N}\circ\a_{2}^{N}](u)\right|\geq m>0,\quad\forall u\in I.
\]

\end{lem}
We remark here that the lower bound is the harder one. The upper bound
follows from the expanding property of $T$ and regularity of $\tau$.

Now suppose we deal with $\tau$ with property (NLI). Let $\xi,\eta,u_{0},v_{0}$
and $j_{0}$ be as in Definition \ref{NLI}. 

\textit{Throughout the rest of this paper, the assignments $N\to\a_{1}^{N}$
and $N\to\a_{2}^{N}$ are fixed as those given by Lemma \ref{NLIcons}. }

We do however need to know some of the details about how the $\a_{i}^{N}$
have been constructed, which we give now.

As in the proof of \cite[Proposition 5.5]{NAUD} there are $\e>0$
and an open interval $\U$ with 
\[
I_{j_{0}}\supset\U\ni u_{0}
\]
such that 
\[
\left|\frac{\partial\varphi_{\xi,\eta}}{\partial u}(u',v_{0})\right|>\e
\]
for all $u'\in\U$. We define for any $n$ 
\[
\beta_{1}^{n}=T_{\xi_{-n+1}}^{-1}\circ\ldots\circ T_{\xi_{0}}^{-1},
\text{ and }
\beta_{2}^{n}=T_{\eta_{-n+1}}^{-1}\circ\ldots\circ T_{\eta_{0}}^{-1},
\]
two branches of $T^{-n}$ on $I_{j_{0}}$. In the proof of \cite[Proposition 5.5]{NAUD},
Naud also constructs 
\[
\psi:I\to\U
\]
which is a branch of $T^{-\hat{p}}$ for some $\hat{p}$ a fixed positive
integer related to the mixing and expanding properties of $T$. The
image of $\psi$ is a disjoint union of $k$ closed intervals each
of which is diffeomorphic to some $I_{j}$ by $\psi$. We denote by
$U_{0}$ the image of $\psi$. We will use the parameterization 
\[
N=\tilde{N}+\hat{p}.
\]
Then the $\a_{i}^{N}$ are defined by 
\[
\a_{i}^{N}=\beta_{i}^{\tilde{N}}\circ\psi.
\]
As $\tilde{p}$ is fixed, $\tilde{N}$ and $N$ are coupled. They
are to be chosen, depending on $b$ and other demands in the following.

\subsection{Construction of Dolgopyat operators}

The following is proved by Naud \cite[Proposition 5.6]{NAUD}.
\begin{prop}[Triadic partition]
\label{triad}There are $A_{1},A'_{1}>0$ and $A_{2}>0$ such that
when $\e>0$ is small enough, there is a finite collection $(V_{i})_{1\leq i\leq Q}$
of closed intervals ordered along $U_{0}$ such that: 
\begin{enumerate}
\item $\U\supset\cup_{i=1}^{Q}V_{i}\supset U_{0}$, $V_{i}\cap\Int U_{0}\neq\emptyset$
for all $i$ and $\Int V_{i}\cap\Int V_{j}=\emptyset$ when $i\neq j$. 
\item \label{second} For all $1\leq i\leq Q$, $\e A'_{1}\leq|V_{i}|\leq\e A_{1}$. 
\item For all $1\leq j\leq Q$ with $V_{j}\cap K\neq\emptyset$, either
$V_{j-1}\cap K\neq\emptyset$ and $V_{j+1}\cap K\neq\emptyset$ or
$V_{j-2}\cap K\neq\emptyset$ and $V_{j-1}\cap K\neq\emptyset$ or
$V_{j+1}\cap K\neq\emptyset$ and $V_{j+2}\cap K\neq\emptyset$ .
In other words, intervals that intersect $K$ come at least in triads. 
\item \label{fourth} For all $1\leq i\leq Q$ with $V_{i}\cap K\neq\emptyset$,
$V_{i}\cap K\subset U_{0}$ and $\mathrm{dist}(\partial V_{i},K)\geq A_{2}|V_{i}|$. 
\end{enumerate}
\end{prop}
Now following Naud we can construct the Dolgopyat operators. Suppose
that we are working at frequency $s=a+ib$. Then for fixed $\e'$
to be chosen, we construct a triadic partition $(V_{i})_{i=1}^{Q}$
of $U_{0}$ with $\e=\e'/|b|$ as in Proposition \ref{triad}. Then
for all $i\in\{1,2\}$ and $j\in\{1,\ldots,Q\}$ we set 
\[
Z_{j}^{i}=\beta_{i}^{\tilde{N}}(V_{j}\cap U_{0}).
\]
We will write 
\[
X_{j}=\{x\in I:\psi(x)\in V_{j}\},\quad1\leq j\leq Q.
\]
Properties \ref{fourth} and \ref{second} of Proposition \ref{triad}
imply that 
\begin{equation}
\mathrm{dist}(K\cap V_{j},\partial V_{j})\geq A_{2}|V_{j}|\geq\frac{A_{2}A'_{1}\e'}{|b|}.\label{eq:cutoffest}
\end{equation}
whenever $K\cap V_{j}\neq0$. For such $j$ we can find a $C^{1}$
cutoff $\chi_{j}$ on $I$ that is $\equiv1$ on the convex hull of
$K\cap V_{j}$ and $\equiv0$ outside $V_{j}$. Due to \eqref{eq:cutoffest}
we can ensure that 
\[
|\chi'_{j}|\leq A_{3}\frac{|b|}{\e'},\quad A_{3}=A_{3}(A_{2},A'_{1}).
\]
Then the index set $\I_{s}$ is defined to be 
\[
\I_{s}:=\{(i,j)\::\:1\leq i\leq2,1\leq j\leq Q,V_{j}\cap K\neq\emptyset\}.
\]
Allow $0<\theta<1$ to be fixed shortly. For all $J\subset\I_{s}$
we define $\chi_{J}\in C^{1}(I)$ by 
\[
\chi_{J}(x)=\begin{cases}
1-\theta\chi_{j}(\psi(T^{N}x)), & \text{if \ensuremath{x\in Z_{i}^{j}} for \ensuremath{(i,j)\in J}}.\\
1, & \text{else}.
\end{cases}
\]
Then the Dolgopyat operators on $C^{1}(I)$ are defined by 
\[
\N_{s}^{J}(f)=L_{a}^{N}(\chi_{J}f).
\]
Recall that $L_{a}$ is the transfer operator at $s=a$.

Let us return to our Lemma \ref{core} so that we can complete our
definitions.
\begin{defn}
\label{dense} We say that $J\subset I_{s}$ is dense if for all $1\leq j\leq Q$
with $V_{j}\cap K\neq\emptyset$ there is some $1\leq j'\leq Q$ with
$(i,j')\in J$ for some $i\in\{1,2\}$ and with $|j-j'|\leq2.$ 
\end{defn}
We define $\E_{s}$ of Lemma \ref{core} to be the set of $J\subset\I_{s}$
such that $J$ is dense.

The following is proved in \cite{NAUD} - we have tried to contain
everything that we use as a black box here.
\begin{prop}[Naud]
 \label{blackbox} There are constants $a_0$, $b_0$, $A$, $N_0$ such that for each sufficiently small $\e'$ there is $\theta_0(\e')$ and $\rho(\e')$
 such that
 when $N>N_0$, $\theta<\theta_0(\e')$,  $|a-\delta|<a_0$ and $|b|>b_0$,
Properties \ref{corefirst} and \ref{coresecond} of Lemma \ref{core} hold for
our $(N,|b|,\theta,\e')$ parameterized and $\E_{s}$-indexed Dolgopyat
operators with respect to this $\rho$.

Furthermore,  there is positive $C_0$ such that when $|a-\delta|<a_0$ we have for arbitrary $N$ 
\begin{equation}
|(\tau_{a}^{N}\circ\a^{N})'(x)|\leq C_0 \label{eq:perturbbound},
\end{equation}
and when $N>N_0$, $b>b_0$ we have
\begin{equation}
|([\tau_{a}^{N}+ib\tau^{N}]\circ\a^{N})'(x)|\leq  \frac{A|b|}{4}.\label{eq:perturbbound2}
\end{equation}
This was a factor in how $A$ was chosen.
\end{prop}
The proof of the inequalities  above are discussed in \cite[pg. 137]{NAUD}.

\textit{This fully completes the definition of the Dolgopyat operators
modulo choice of $\e'$, $\theta$ and $N$ - the $A$ and $\rho$
required for Lemma \ref{core} are that specified by Proposition \ref{blackbox}
given these parameters.}

\subsection{Proof of Lemma \ref{core}, property \ref{corethird}}

Our remaining task in this section is to prove property \ref{corethird}
of Lemma \ref{core}. This is proved for complex valued functions
by Naud in \cite[pp. 140-144]{NAUD}. Naud makes some use of taking
quotients of values of functions that we will have to work around.

We give the details now. Recall that $\e',\theta$ are still undetermined.
The following technical Lemma is the vector valued version of \cite[Lemma 5.10]{NAUD}.
Recall that $c_{q}:I\to U(\C^{\G_{q}})$ is our twisting unitary valued
map at level $q$. We will need to consider  the quantity $c_{q}^{N}(\a_{i}^{N}x)$, defined as in the
Dictionary of page \pageref{table:dictionary}, 
where $\a_{i}^{N}$, $i=1,2$, are the two particular branches of $T^{-N}$
that are given by Lemma \ref{NLIcons}. We record the key fact here
that since $c_{q}$ is locally constant, so too is $c_{q}^{N}$ for
any $N$.
\begin{lem}[Key technical fact towards non-stationary phase]
\label{technical}Let $H\in\CC_{A|b|}$, $f\in C^{1}(I;\C^{\G_{q}})$
such that $|f|\leq H$ and $|f'|\leq A|b|H$. For $i=1,2,$ define
for $\theta$ a small real parameter and for any $q$, 
\[
\Theta_{1}(x):=\frac{|e^{[\tau_{a}^{N}+ib\tau^{N}](\a_{1}^{N}x)}c_{q}^{N}(\a_{1}^{N}x)f(\a_{1}^{N}x)+e^{[\tau_{a}^{N}+ib\tau^{N}](\a_{2}^{N}x)}c_{q}^{N}(\a_{2}^{N}x)f(\a_{2}^{N}x)|}{(1-2\theta)e^{\tau_{a}^{N}(\a_{1}^{N}x)}H(\a_{1}^{N}x)+e^{\tau_{a}^{N}(\a_{2}^{N}x)}H(\a_{2}^{N}x)};
\]
\[
\Theta_{2}(x):=\frac{|e^{[\tau_{a}^{N}+ib\tau^{N}](\a_{1}^{N}x)}c_{q}^{N}(\a_{1}^{N}x)f(\a_{1}^{N}x)+e^{[\tau_{a}^{N}+ib\tau^{N}](\a_{2}^{N}x)}c_{q}^{N}(\a_{2}^{N}x)f(\a_{2}^{N}x)|}{e^{\tau_{a}^{N}(\a_{1}^{N}x)}H(\a_{1}^{N}x)+(1-2\theta)e^{\tau_{a}^{N}(\a_{2}^{N}x)}H(\a_{2}^{N}x)}.
\]
Then for $N$ large enough, one can choose $\theta$ and $\e'$ small
enough such that for $j$ with $X_{j}\cap K\neq\emptyset$, there
are $j'$ with $|j-j'|\leq2$, $X_{j'}\cap K\neq\emptyset$ and $i\in\{1,2\}$
such that 
\begin{equation}
\Theta_{i}(x)\leq1\quad\text{for all \ensuremath{x\in X_{j'}}.}
\end{equation}

\end{lem}
Before giving the proof we must state a simple Lemma from \cite{NAUD}.
The proof goes through easily in our vector valued setting. This is
also covered in \cite[Lemma 3.29]{OW}.
\begin{lem}[Lemma 5.11 of \cite{NAUD}]
\label{alt} Let $Z\subset I$ be an interval with $|Z|\leq c/|b|$.
Let $H\in\CC_{A|b|}$ and $f\in C^{1}(I;\C^{\G_{q}})$ with $|f|\leq H$
and $|f'|\leq A|b|H$. Then for $c$ small enough, we have either
\[
|f(u)|\leq\frac{3}{4}H(u)\quad\text{for all \ensuremath{u\in Z}, or}
\]
\[
|f(u)|\geq\frac{1}{4}H(u)\quad\text{for all \ensuremath{u\in Z}.}
\]

\end{lem}
We also need the following piece of trigonometry from \cite[Lemma 5.12]{NAUD}.
\begin{lem}[Sharp triangle inequality]
\label{trig}Let $V$ be a finite dimensional complex vector space
with Hermitian inner product $\langle\bullet,\bullet\rangle$. For
non zero vectors $z_{1},z_{2}$ with $|z_{1}|/|z_{2}|\leq L$ and
\begin{equation}
\Re\langle z_{1},z_{2}\rangle\leq(1-\eta)|z_{1}||z_{2}|,\label{eq:subconvex}
\end{equation}
there is $\delta=\delta(L,\eta)$ such that 
\[
|z_{1}+z_{2}|\leq(1-\delta)|z_{1}|+|z_{2}|.
\]
\end{lem}
We remark that while Lemma \ref{trig} is elementary, the fact that there is no dependence on the dimension of $V$
is one of the crucial points in our arguments.
\begin{proof}[Proof of Lemma \ref{technical}]
Choose $\e'$ small enough so that Lemma \ref{alt} holds for all
$Z=Z_{j}^{i}$ (with $c=\e'$). As in \cite{NAUD} by choosing $N$
large enough it is possible to assume $|Z_{j}^{i}|\leq|V_{j}|$ for
all $j,i$. We also enforce $\theta<1/8$ so that $1-2\theta\geq3/4$.

Now let $V_{j},V_{j+1},V_{j+2}$ all have non empty intersection with
$K$. One of the $j,j+1,j+2$ will be the $j'$ of the Lemma. Set
$\Xhat_{j}=X_{j}\cup X_{j+1}\cup X_{j+2}$ and assume as in Naud that
$\Xhat_{j}$ is contained in one connected component of $I$; note
that $\Xhat_{j}$ is connected.

Following from our choice of $\theta$, if there is $j'\in\{j,j+1,j+2\}$
and $i\in\{1,2\}$ with $|f(u)|\leq\frac{3}{4}H(u)$ when $u\in Z_{j'}^{i}$
then $\Theta_{i}(u)\leq1$ on $Z_{j'}^{i}$ and we are done. So we
can assume $|f(u)|>\frac{3}{4}H(u)$ for some $u$ in each $Z_{j'}^{i}$.
Hence by Lemma \ref{alt}, for all $i,j'$ we have 
\begin{equation}
|f(u)|\geq\frac{1}{4}H(u)>0,\quad\forall u\in Z_{j'}^{i}.\label{eq:assumption}
\end{equation}

We make the definition 
\[
z_{i}(x):=\exp\left([\tau_{a}^{N}+ib\tau^{N}](\a_{i}^{N}x)\right)c_{q}^{N}(\a_{i}^{N}x)f(\a_{i}^{N}x),\quad z_{i}:\Xhat_{j}\to\C^{\G_{q}},\quad i=1,2.
\]
The result follows from Lemma \ref{trig} after establishing bounds
on the relative size and angle of $z_{1},z_{2}$ uniformly in appropriate
$X_{j'}$.

\textbf{Control of relative size.} Firstly we wish to control the
relative size of $z_{1},z_{2}$. This is done by Naud and his estimates
go through directly in our case, after making all substitutions of
the form 
\[
\left|\frac{z_{1}(x)}{z_{2}(x)}\right|\to\frac{|z_{1}(x)|}{|z_{2}(x)|}
\]
and bearing in mind that $c_{q}^{N}$ is a unitary valued function.
This caters to our inability to divide non zero vectors. The output
of Naud's argument in \cite[pp. 141-142]{NAUD} is that given $j'\in\{j,j+1,j+2\}$,
either $|z_{1}(x)|\leq M|z_{2}(x)|$ for all $x\in X_{j'}$ or $|z_{2}(x)|\leq M|z_{1}(x)|$
for all $x\in X_{j'}$, where 
\[
M=4\exp(2NB_{a})\exp(2A\e'A_{1})
\]
and 
\[
B_{a}=a\|\tau\|_{\infty}+|P(-a\tau)|+2\|\log h_{a}\|_{\infty}
\]
is a locally bounded function that arises in the estimation of $\tau_{a}^{N}$
(cf. \cite[pg. 139]{NAUD}). Returning to the overall argument, this
means that we are done when we can establish \eqref{eq:subconvex}
with some $\eta$ uniformly on some $X_{j'}$. 

\textbf{Control of relative angle.} The key argument here is to very
carefully control the angles between the functions $z_{1}$ and $z_{2}$.
One sets 
\[
\Phi(x):=\frac{\langle z_{1}(x),z_{2}(x)\rangle}{|z_{1}(x)||z_{2}(x)|}
\]
which is the same as 
\[
\Phi(x)=\exp(ib(\tau^{N}(\a_{1}^{N}x)-\tau^{N}(\a_{2}^{N}x)))\frac{\langle c_{q}^{N}(\a_{1}^{N}x)f(\a_{1}^{N}x),c_{q}^{N}(\a_{2}^{N}x)f(\a_{2}^{N}x)\rangle}{|f(\a_{1}^{N}x)||f(\a_{2}^{N}x)|}.
\]
Define 
\[
u_{i}(x)=c_{q}^{N}(\a_{i}^{N}x)\frac{f(\a_{i}^{N}x)}{|f(\a_{i}^{N}x)|},\quad x\in\Xhat_{j},\quad i=1,2.
\]
Then the $u_{i}$ are $C^{1}$ as $f$ is non vanishing through \eqref{eq:assumption}.
We have 
\[
(c_{q}^{N}.f)\circ\a_{i}^{N}=|f\circ\a_{i}^{N}|.u_{i},
\]
so that, differentiating on both sides and using $(c_{q}^{N})'\equiv0$,
\[
(c_{q}^{N}\circ\a_{i}^{N}).(f\circ\a_{i}^{N})'=|f\circ\a_{i}^{N}|'u_{i}+|f\circ\a_{i}^{N}|u_{i}'.
\]
As $u_{i}$ has constant length 1 it follows that $u_{i}$ and $u'_{i}$
are orthogonal (in $\R^{2|\G_{p}|}$). Therefore 
\[
|[f\circ\a_{i}^{N}]'|^{2}=(|f\circ\a_{i}^{N}|')^{2}+|f\circ\a_{i}^{N}|^{2}|u'_{i}|^{2}.
\]
It now follows that 
\[
|u'_{i}(x)|\leq\frac{|[f\circ\a_{i}^{N}]'(x)|}{|f(\a_{i}^{N}x)|}.
\]
We estimate the right hand side by a direct calculation using the
chain rule with the expanding property of $T$ and our assumptions
on $H$ from \eqref{eq:assumption} and the hypotheses of Lemma \ref{technical}.
Indeed, Naud performs a similar calculation \cite[pg. 142]{NAUD}
which yields 
\begin{equation}
|u'_{i}(x)|\leq8A|b|\frac{D}{\g^{N}}.\label{eq:uest}
\end{equation}
Note that we can rewrite the central quantity $\Phi$ as 
\begin{equation}
\Phi(x)=\exp(ib(\tau^{N}(\a_{1}^{N}x)-\tau^{N}(\a_{2}^{N}x)) )\langle u_{1}(x),u_{2}(x)\rangle.\label{eq:phisimple}
\end{equation}
We can use \eqref{eq:uest} and Cauchy-Schwarz to get 
\begin{equation}
\left|\frac{d}{dx}\langle u_{1},u_{2}\rangle\right|=\left|\langle u'_{1},u_{2}\rangle+\langle u_{1},u'_{2}\rangle\right|\leq16A|b|\frac{D}{\g^{N}}.\label{eq:bracketest}
\end{equation}
Note that we have the diameter bound 
\begin{equation}
\diam(\Xhat_{j})\leq3A_{1}\frac{\e'}{|b|}\|(\psi^{-1})'\|_{\infty}\label{eq:diameter}
\end{equation}
so that using \eqref{eq:bracketest} we have 
\[
|\langle u_{1}(x_{1}),u_{2}(x_{1})\rangle-\langle u_{1}(x_{2}),u_{2}(x_{2})\rangle|\leq3\cdot16\cdot AA_{1}\|(\psi^{-1})'\|_{\infty}\e'\frac{D}{\g^{N}}
\]
for any $x_{1},x_{2}\in\Xhat_{j}$; note here that the cocycles $c_{q}^{N}(\alpha_{i}^{N}x)$
are constant on $\hat{X}_{j}$. We now enforce $\e'<1/10$ and $N$
large enough so that 
\[
48\cdot AA_{1}\|(\psi^{-1})'\|_{\infty}\frac{D}{\g^{N}}<1.
\]

Let us cut off one branch of reasoning. Suppose that there is $x_{0}\in\Xhat_{j}$
with 
\[
|\langle u_{1}(x_{0}),u_{2}(x_{0})\rangle|<1/10.
\]
Then for all $x\in\Xhat_{j}$ we have 
\[
|\langle u_{1}(x),u_{2}(x)\rangle|<1/5.
\]
It would follow that $|\Re\Phi(x)|<1/5$ for all $x\in\Xhat_{j}$
and the Lemma would be proved by our argument with trigonometry.

Therefore we can now assume 
\[
|\langle u_{1}(x),u_{2}(x)\rangle|\geq1/10
\]
for all $x\in\Xhat_{j}$. Then the new function 
\[
U(x)=\frac{\langle u_{1}(x),u_{2}(x)\rangle}{|\langle u_{1}(x),u_{2}(x)\rangle|}\in\C
\]
is $C^{1}$ on $\Xhat_{j}$ of constant length $1$ and by an argument
we have made before 
\begin{equation}
|U'(x)|\leq\frac{|\langle u_{1},u_{2}\rangle'(x)|}{|\langle u_{1}(x),u_{2}(x)\rangle|}\leq10\cdot16\cdot A|b|\frac{D}{\g^{N}},\label{eq:derivU}
\end{equation}
using \eqref{eq:bracketest}. We can write 
\[
U(x)=\exp(i\phi(x))
\]
for some $C^{1}$ real valued $\phi:\Xhat_{j}\to\R$. Then \eqref{eq:derivU}
reads 
\begin{equation}
|\phi'(x)|\leq160A|b|\frac{D}{\g^{N}}.\label{eq:phideriv}
\end{equation}

As we assume $\Phi\neq0$ on $\Xhat_{j}$, we can find a $C^{1}$
function that we will denote 
\[
\arg\Phi:\Xhat_{j}\to S^{1}=\R/2\pi\Z,\quad\Phi(x)=\exp(i\arg\Phi(x))\cdot|\Phi(x)|.
\]

Now define 
\[
F(x)=(\tau^{N}(\a_{1}^{N}x)-\tau^{N}(\a_{2}^{N}x)),\quad x\in\Xhat_{j}.
\]
The critical output of the (NLI) property for $\tau$, Lemma \ref{NLIcons},
tells us that 
\begin{equation}
0<m\leq|F'(x)|\leq m'\label{eq:Fderiv}
\end{equation}
when we choose $N>N_{0}$, which we do. As 
\[
\arg\Phi=bF+\phi
\]
we now have, incorporating \eqref{eq:Fderiv} and \eqref{eq:phideriv}
\[
|b|(m-10\cdot16A\frac{D}{\g^{N}})\leq|(\arg\Phi)'|\leq|b|(m'+10\cdot16A\frac{D}{\g^{N}}).
\]
We fix, finally, $N$ large enough so that we gain $C_{2}>C_{1}>0$
(depending only on $N$, $m$, $m'$, $A$, $D$, and $\g$) with

\[
|b|C_{1}\leq|(\arg\Phi)'|\leq|b|C_{2}.
\]
Now by estimating diameters of $X_{j+1}$ and $\Xhat_{j}$ from Proposition
\ref{triad} together with the mean value theorem, the total cumulative
change of argument of $\Phi$ between $x_{j}\in X_{j}$ and $x_{j+2}\in X_{j+2}$,
written $\Delta$, is between

\[
C_{3}\e'\leq\Delta\leq C_{4}\e'
\]
where 
\[
C_{3}=C_{1}A'_{1}\inf_{U_{0}}|(\psi^{-1})'|>0,\quad C_{4}=C_{2}3A_{1}\|(\psi^{-1})'\|_{\infty}.
\]
We now enforce $\e'<\pi/(2C_{4})$ so that we no longer need to worry
about $\arg\Phi$ winding around the circle. We are about to conclude.
Now $\e'$ is fixed. By our trigonometric strategy, we are done with
\[
\theta=\delta\left(M,\left(\frac{C_{3}\e'}{100}\right)^{2}\right)
\]
unless there exist $x_{j}\in X_{j}$ and $x_{j+2}\in X_{j+2}$ with
\[
\Re\Phi(x_{k})>1-\left(\frac{C_{3}\e'}{100}\right)^{2},\quad k=j,j+2.
\]
In this case, by the Schwarz inequality we know 
\[
|\Phi(x_{k})|\leq1\quad k=j,j+2
\]
so it follows that now using the principal branch for $\arg$ and
e.g. $|\sin x|\leq2|x|$ 
\[
|\arg\Phi(x_{k})|\leq C_{3}\e'/50,\quad k=j,j+2.
\]
Given that the argument of $\Phi$ moves at least by $C_{3}\e'$ in
one direction between $x_{j}$ and $x_{j+2}$ and does not move more
than $\pi/2$ (hence does not wind), this is a contradiction.
\end{proof}
We can now conclude this section with 
\begin{proof}[Proof of Lemma \ref{core}, property \ref{corethird}]
Choose $N$, $\theta$ and $\e'$ so that Proposition \ref{blackbox}
holds as well as Lemma \ref{technical}. Increasing $N$ if necessary
we may also assume that $\frac{D}{\gamma^{N}}\leq\frac{1}{4}$.

Suppose we are given $H\in\CC_{A|b|}$ and $f\in C^{1}(I;\C^{\G_{q}})$
such that $|f|\leq H$ and $|f'|\leq A|b|H$. The second inequality
stated in property \ref{corethird} is softer so we prove this first.
The complex scalar version of this inequality is proved in \cite[pg. 138]{NAUD}.

We calculate 
\[
[L_{s,q}^{N}f](x)=\sum_{\a^{N}}\exp([\tau_{a}^{N}+ib\tau^{N}](\a^{N}x))c_{q}^{N}(\a^{N}x)f(\a^{N}x).
\]
where 
\[
c_{q}^{N}(y)=c_{q}(T^{N-1}y)\ldots c_{q}(Ty).c_{q}(y)
\]
and the sum is over branches of $T^{-N}$. Therefore 
\begin{align*}
[L_{s,q}^{N}f]'(x) & =\sum_{\a^{N}}([\tau_{a}^{N}+ib\tau^{N}]\circ\a^{N})'(x)\exp([\tau_{a}^{N}+ib\tau^{N}](\a^{N}x))c_{q}^{N}(\a^{N}x)f(\a^{N}x)\\
 & +\sum_{\a^{N}}\exp([\tau_{a}^{N}+ib\tau^{N}](\a^{N}x))c_{q}^{N}(\a^{N}x)(f\circ\a^{N})'(x),
\end{align*}
$c_{q}^{N}$ being locally constant. Using that $c_{q}^{N}$ is unitary
and bounding derivatives of $\a^{N}$ with the eventually expanding
property and chain rule gives 
\begin{align*}
|[L_{s,q}^{N}f]'(x)|\leq\sum_{\a^{N}}|([\tau_{a}^{N}+ib\tau^{N}]\circ\a^{N})'(x)|\exp([\tau_{a}^{N}](\a^{N}x))H(\a^{N}x)\\
+\frac{D}{\g^{N}}\sum_{\a^{N}}\exp([\tau_{a}^{N}](\a^{N}x))A|b|H(\a^{N}x).
\end{align*}
Using the inequality \eqref{eq:perturbbound2} in Proposition \ref{blackbox} and our choice
of $N$ we get 
\[
|[L_{s,q}^{N}f]'(x)|\leq\frac{1}{2}A|b|[L_{a}^{N}H](x)\leq A|b|[\N_{s}^{J}H](x)
\]
given the very mild assumption $\theta<1/2$.

Now we turn to the more difficult first inequality of Lemma \ref{core},
property \ref{corethird}. Given that we have established Lemma \ref{technical}
in the vector valued setting, the proof follows by the same argument
as in \cite[pg. 143]{NAUD}. We give the details here for completeness.

Let $J$ be the set of indices $(i,j)$ where $\Theta_{i}(x)\leq1$
when $x\in X_{j}$. The statement of Lemma \ref{technical} is precisely
that this set of indices is dense (recall Definition \ref{dense})
and hence $J\in\E_{s}$ as required. We will prove 
\[
|L_{s,q}^{N}f|\leq\N_{s}^{J}H=L_{a}(\chi_{J}H).
\]
Fix $x$. Notice that if $x\notin\Int X_{j}$ for any $j$ then for
all branches $\a^{N}$ of $T^{-N}$, $\a^{N}x\notin Z_{j}^{i}$ and
so $\chi_{J}(\a^{N}x)=1$ for any $J$. More generally if $x\notin\Int X_{j}$
for any $j$ appearing as a coordinate in $J$ then $\chi_{J}(\a^{N}x)=1$.
Therefore 
\[
|[L_{s,q}^{N}f](x)|\leq\sum_{\a^{N}}\exp(\tau_{a}^{N}(\a^{N}x))H(\a^{N}x)=\N_{s}^{J}[H](x).
\]

We are left to consider $x,J$ such that $x\in\Int(X_{j})$ and $J$
contains $(i,j)$ for some $i$.

Suppose that $(i,j)=(1,j)$ and $(2,j)\notin J$. Then for $\a^{N}\neq\a_{1}^{N}$
a branch of $T^{-N}$, $\chi_{J}(\a^{N}x)=1$ (the only other possibility
would have been $\a^{N}=\a_{2}^{N}$). Then using $\Theta_{1}(x)\leq1$
gives 
\begin{align*}
|L_{s,q}^{N}[f](x)| & \leq\sum_{\a^{N}\neq\a_{1}^{N},\a_{2}^{N}}\exp(\tau_{a}^{N}(\a^{N}(x))H(\a^{N}(x))\\
 & +(1-2\theta)\exp(\tau_{a}^{N}(\a_{1}^{N}(x))H(\a_{1}^{N}(x))+\exp(\tau_{a}^{N}(\a_{2}^{N}(x))H(\a_{2}^{N}(x))\\
 & \leq\N_{s}^{J}[H](x).
\end{align*}
The case $(i,j)=(2,j)$ and $(1,j)\notin J$ is treated the same way.
Finally, if $(1,j)$ and $(2,j)$ are in $J$ then $\Theta_{1}(x),\Theta_{2}(x)\leq1$
from which one can estimate 
\begin{align*}
 & |\exp([\tau_{a}^{N}+ib\tau^{N}](\a_{1}^{N}x))f(\a_{1}^{N}x)+\exp([\tau_{a}^{N}+ib\tau^{N}](\a_{2}^{N}x))f(\a_{2}^{N}x)|\\
 & \leq(1-\theta)\exp(\tau_{a}^{N}(\a_{1}^{N}(x))H(\a_{1}^{N}(x))+(1-\theta)\exp(\tau_{a}^{N}(\a_{2}^{N}(x))H(\a_{2}^{N}(x))\\
 & \leq\exp(\tau_{a}^{N}(\a_{1}^{N}(x))\chi_{J}(\a_{1}^{N}x)H(\a_{1}^{N}(x))+\exp(\tau_{a}^{N}(\a_{2}^{N}(x))\chi_{J}(\a_{2}^{N}x)H(\a_{2}^{N}(x)).
\end{align*}
Also noting that $\chi_{J}(\a^{N}x)=1$ when $\a^{N}\neq\a_{i}^{N}$,
$i=1,2$, the previous inequality shows 
\[
|L_{s,q}^{N}[f](x)|\leq\N_{s}^{J}[H](x)
\]
in our final remaining case. The proof is complete. 
\end{proof}

\section{Bounds for transfer operators: small imaginary part}

\label{smallimaginarypart}

In this section we prove Part \ref{enu:mainmodular} of Theorem \ref{thm:maintransfer}. The key point is to think of $W\in C^1(I,\C^{\G_q})$ as a function on $I\times \G_q$ and decouple the variables.
This allows us to relate the transfer operator to a convolution operator on $\C^{\G_q}$. The relevant 
convolution operators have good spectral radius bounds that stem from the expander theory of
$\G_q$ as described in the Appendix -- the expansion technology requires 
that  we restrict $q$ to be coprime to a finite bad modulus $Q_0$, we make this restriction throughout. We now begin decoupling arguments in order to relate Part \ref{enu:mainmodular} of Theorem \ref{thm:maintransfer}  to the main result of the Appendix that we state as Theorem \ref{thm:mainbody} below.

\subsection{Accessing the convolution.}

\label{expingredients} We define $E_{q}$ to be the space of functions
of $\G_{q}=\SL_{2}(\Z/q\Z)$ that are orthogonal to all functions
lifted from $\G_{q'}$ for $q'|q$. We set out to show that when we
iterate $L_{s,q}^{n}$ we suitably contract the $C^{1}$ norm.

We have calculated already that for $W\in C^{1}(I,\C^{\G_q})$ with $\|W\|_{C^{1}}<\infty$
and taking on values only in the orthocomplement to constant functions
\begin{equation}
[L_{s,q}^{N}W](x)=\sum_{\a^{N}}\exp([\tau_{a}^{N}+ib\tau^{N}](\a^{N}x))c_{q}^{N}(\a^{N}x)W(\a^{N}x)\label{eq:LNWexplicit}
\end{equation}
where the sum is over branches of $T^{-N}$ on the interval containing
$x$. We can write these branches in a special form. 
They are given precisely by sequences
\[
\a^{N}=g_{i_{1}}g_{i_{2}}\ldots g_{i_{N}}
\]
where the $g_{i_j}$ form an admissible sequence.
If in the
general Schottky semigroup setting, we also require that if $x\in \tilde{I}_{i}$ for $1\leq i \leq 2k'$ then  $i_N\neq i+k' \bmod 2k'$, recalling the notation of   Section \ref{sub:A-dynamical-system}.

It will be convenient to make the parametrization 
\[
N=M+R,\quad M,R>0.
\]
We then write 
\[
\a^{N}=\a^{M}\a^{R}
\]
where 
\begin{equation}
\a^{M}=g_{i_{1}}\ldots g_{i_{M}},\quad\alpha^{R}=g_{i_{M+1}}\ldots g_{i_{N}},\label{eq:splitting}
\end{equation}
and view these as globally defined maps on $I.$ Then $\alpha^{N}$
is uniquely parameterized by the two sequences appearing in \eqref{eq:splitting}.
When we write $\alpha^{M}$ and $\alpha^{R}$ henceforth we always
mean compositions of these forms. Notice that the choice of $\alpha^R$ is restricted depending on $x$ and $\alpha^M$ is restricted depending on 
$g_{i_{M+1}}$.

For each of the intervals $I_i$ we pick a point $x_0(i)\in I_i$. 
For each $\alpha^M$ we pick $i_0=i_0(\alpha^M)$ such that $\alpha^M$ gives a well defined branch on $I_{i_0}$. Then
\[
d(\a^{N}x,\a^{M}x_{0}(i_0))=d(\a^{M}(\a^{R}x),\a^{M}x_{0}(i_0))\leq\frac{D}{\g^{M}}\diam(I)
\]
by the eventually expanding property of $T$. Then 
\[
\|W(\a^{N}x)-W(\a^{M}x_{0}(i_0))\|\leq\frac{D}{\g^{M}}\diam(I)\|W\|_{\Lip}.
\]
It follows then that 
\begin{align*}
[L_{s,q}^{N}W](x) & =\sum_{\alpha^M} \sum^*_{\a^{R}}\exp([\tau_{a}^{N}+ib\tau^{N}](\a^{N}x))c_{q}^{N}(\a^{N}x)W(\a^{M}x_{0}(i_0))\\
 & +O\left(\|W\|_{\Lip}\frac{D}{\g^{M}}\diam(I)\sum_{\a^{N}}\exp(\tau_{a}^{N}(\a^{N}x))\right)
\end{align*}
where the star on summation means that
 we restrict to those $\alpha^R$ with necessary restriction on $g_{i_N}$ coming from $x$ and $g_{i_{M+1}}$ coming from $\alpha^M$. Note that $i_0$ depends on $\alpha^M$.
We will assume that $D\g^{-M}$ is small, say $<1/(100\diam(I))$
and note that the sum in the error term is 
\[
\sum_{\a^{N}}\exp(\tau_{a}^{N}(\a^{N}x))=L_{a}^{N}[1](x)=1(x)=1
\]
as the operator has been normalized. So then 
\begin{equation}
[L_{s,q}^{N}W](x)=\sum_{\alpha^M} \sum^*_{\a^{R}}\exp([\tau_{a}^{N}+ib\tau^{N}](\a^{N}x))c_{q}^{N}(\a^{N}x)W(\a^{M}x_{0}(i_0))+O(\|W\|_{\Lip}\g^{-M}).\label{eq:uncoupled}
\end{equation}

This is an important estimate as it allows us access the expansion
properties coming from $c_{q}$ by decoupling $M$ and $N$.

Recall that $c_{q}$ was obtained by reducing the matrices $g_{i}$
modulo $q$ to obtain a locally constant mapping $c_{q}:I\to\G_{q}$.
This mapping can be reinterpreted as a unitary valued map $c_{q}:I\to U(\C^{\G_{q}})$
via the right regular representation of $\G_{q}$.

For any specified $\a^{M}$ as in \eqref{eq:splitting} and $x\in I$ we construct the complex
valued measure on $\G_{q}$ 

\[
\mu_{s,x,\a^{M}}=\sum^*_{\alpha^{R}}\exp([\tau_{a}^{N}+ib\tau^{N}](\a^{M}\alpha^{R}x))\delta_{c_{q}^{R}(\a^{R}x)^{-1}}\label{eq:mudefinition}
\]
 where $\delta_{g}$ gives mass one to $g\in\G_{q}$.  We note for the
reader's convenience that one can calculate from the Definition in
the Dictionary of page \pageref{table:dictionary}, Section \ref{sec:Counting} 

\[
c_{q}^{R}(\alpha^{R}x)=g_{i_{N}}g_{i_{N-1}}\ldots g_{i_{M+1}}\bmod q,\quad c_{q}^{M}(\alpha^{M}x_{0}(i_0))=g_{i_{M}}g_{i_{M-1}}\ldots g_{i_{1}}\bmod q.
\]

For any $f\in C^{1}(I;\C^{\G_{q}})$ and $\a^{M}$ as in \eqref{eq:splitting}
we construct a complex valued measure $\varphi_{f,\a^{M}}$ by 
\[
\varphi_{f,\a^{M}}=\sum_{g\in\G_{q}}f(\a^{M}x_{0}(i_0))\lvert_{g}\:\delta_{gc_{q}^{M}(\a^{M}x_{0}(i_0))^{-1}}
\]
where $c_{q}$ is thought of as $\G_{q}$ valued and $f(\a^{M}x_{0}(i_0))$
thought of as a $\C$-valued function on $\G_{q}$, with $\lvert_{g}$
standing for evaluation at $g$. Also recall $i_0=i_0(\alpha^M)$. Then
\begin{align*}
[\varphi_{f,\a^{M}}\star \mu_{s,x,\a^{M}}] & =\sum_{g\in\G_{q}}\sum_{\alpha^{R}}^*\exp([\tau_{a}^{N}+ib\tau^{N}](\alpha^{M}\alpha^{R}x))f(\a^{M}x_{0}(i_0))\lvert_{g} \delta_{gc_{q}^{M}(\a^{M}x_{0}(i_0))^{-1}} \star  \delta_{c_{q}^{R}(\a^{R}x)^{-1}}\\
 & =\sum_{g\in\G_{q}}\sum_{\a^{R}}^*\exp([\tau_{a}^{N}+ib\tau^{N}](\alpha^{M}\alpha^{R}x))f(\a^{M}x_{0}(i_0))\lvert_{g}\delta_{gc_{q}^{N}(\alpha^{M}\a^{R}x)^{-1}}.
\end{align*}
This means that, now viewed as a function on $\SL_{2}(\Z/q\Z)$ 
\[
[\varphi_{f,\a^{M}}\star \mu_{s,x,\a^{M}}]=\sum_{\a^{R}}^*\exp([\tau_{a}^{N}+ib\tau^{N}](\a^{N}x))c_{q}^{N}(\a^{M}\alpha^{R}x)f(\a^{M}x_{0}(i_0)).
\]
The reader should compare this with \eqref{eq:uncoupled}.

\subsection{Bounds for $\mu_{s,x,\protect\a^{M}}$ }
\:

We need a bound for $\|\mu_{s,x,\a^{M}}\|_{1}$ to use the result
of the Appendix. Firstly we write 
\begin{equation}
|\mu_{s,x,\a^{M}}|\leq\sum_{\alpha^{R}}^*\exp(\tau_{a}^{N}(\a^{N}x))\delta_{c_{q}^{R}(\a^{R}x)}.\label{eq:pointwisebound}
\end{equation}
Notice that 
\begin{equation}
\tau_{a}^{N}(\a^{M}\alpha^{R}x)=\tau_{a}^{M}(\a^{M}\alpha^{R}x)+\tau_{a}^{R}(\a^{R}x).\label{eq:splitup}
\end{equation}
Then 
\[
\|\mu_{s,x,\a^{M}}\|_{1}\leq\sum_{\alpha^{R}}^*\exp(\tau_{a}^{M}(\a^{M}\alpha^{R}x))\exp(\tau_{a}^{R}(\a^{R}x)).
\]

We now decouple: let $\a_{0}^{R}$ be any arbitrary choice of $\alpha^{R}$
(a sequence of the $g_{i}$ that is compatible with $\alpha^M$ and $x$). Then 
\[
\tau_{a}^{M}(\a^{M}\alpha^{R}x)-\tau_{a}^{M}(\a^{M}\alpha_{0}^{R}x)=\sum_{n=0}^{M-1}\tau_{a}(T^{n}\a^{N}x)-\tau_{a}(T^{n}\a_{0}^{N}x)
\]
and noting that $T^{n}\a^{M}\alpha^{R}x$ and $T^{n}\a^{M}\alpha_{0}^{R}x$
are within 
\[
\frac{D}{\g^{M-n}}\diam(I)
\]
of one another, we have 
\begin{align}
\tau_{a}^{M}(\a^{M}\alpha^{R}x) & \leq\tau_{a}^{M}(\a^{M}\alpha_{0}^{R}x)+D.\diam(I)\sup_{y\in I}|[\tau_{a}]'(y)|\sum_{n=0}^{M-1}\frac{1}{\g^{M-n}}\\
 & \leq\tau_{a}^{M}(\a^{M}\alpha_{0}^{R}x)+\k_{1}(D,\g,I,\tau,a_{0})\label{align:decouple}
\end{align}
for $|a-\delta|<a_{0}$ (as $\tau_{a}$ is roughly constant in $a$
close to $\delta$). Therefore 
\begin{align*}
\|\mu_{s,x,\a^{M}}\|_{1} & \leq\exp(\k_{1}+\tau_{a}^{M}(\a^{M}\alpha_{0}^{R}x))\sum_{\alpha^{R}}^*\exp(\tau_{a}^{R}(\a^{R}x))\\
 & \leq\exp(\k_{1}+\tau_{a}^{M}(\a^{M}\alpha_{0}^{R}x))[L_{a}^{R}1](x)=\exp(\k_{1}+\tau_{a}^{M}(\a^{M}\alpha_{0}^{R}x)).
\end{align*}
by the normalization of $L_{a}$. We record this bound in the following.
\begin{lem}
\label{mul1}Given $a_{0}$ small enough, there is $\k_{1}=\k_{1}(a_{0})$
such that for all $x$ and $\a^{M}$, 
\[
\|\mu_{s,x,\a^{M}}\|_{1}\leq\exp(\k_{1}+\tau_{a}^{M}(\a^{M}\alpha_{0}^{R}x)),
\]
for $|a-\delta|<a_{0}$. Here $\a_{0}^{R}$ is any admissible choice of $\alpha^{R}$
as in \eqref{eq:splitting} compatible with $\alpha^M$ and $x$.

\end{lem}

We are now in a position to use the main result of the Appendix, which for the convenience of the reader we also state here.
\begin{thm}[Bourgain-Kontorovich-Magee, Appendix]\label{thm:mainbody} There is a finite modulus $Q_0$ and $c >0$ such that when $R\approx c \log q $, $(q,Q_0)=1$, $|a-\delta|<a_{0}$ and $\vp\in E_{q}$, we have
\begin{equation}\label{eq:mainbody}
\| \vp \star \mu_{s,x,\a^{M}} \|_{2} \ \le \ C\
q^{-1/4}\,
B\,
\|\vp\|_{2},
\end{equation}
given that
$$
\|\mu\|_{1}\ <\ B.
$$
\end{thm}
Using Lemma \ref{mul1}, Theorem \ref{thm:mainbody}  now implies that
when $R\approx c\log q$ for suitable $c$ and $|a-\delta|<a_{0}$, for any
$\vp\in E_{q}$ we have 

\begin{equation}
\|\vp\star\mu_{s,x,\a^{M}}\|_{2}\leq Cq^{-1/4}\exp(\k_{1}+\tau_{a}^{M}(\a^{M}\alpha_{0}^{R}x))\|\vp\|_{2}.\label{eq:flattening}
\end{equation}
Then if we use \eqref{eq:uncoupled} we obtain
\begin{align*}
\|[L_{s,q}^{N}W](x)\|_{l^2 (\G_q)} & \leq\sum_{\a^{M}}\|[\varphi_{W,\a^{M}}\star\mu_{s,x,\a^{M}}]\|_{l^{2}(\G_{q})}+O(\|W\|_{\Lip}\gamma^{-M})\\
 & \leq Cq^{-1/4}\exp(\k_{1})\sum_{\a^{M}}\exp(\tau_{a}^{M}(\alpha^{M}\a_{0}^{R}x))\|\varphi_{W,\a^{M}}\|_{l^{2}(\G_{q})}+O(\|W\|_{\Lip}\gamma^{-M}).
\end{align*}
We have now chosen some $\alpha_{0}^{R}$ and we are assuming the
previous conditions on $|a-\delta|<a_{0}$. Since trivially 
\[
\|\varphi_{W,\a^{M}}\|_{l^{2}(\G_{q})}\leq\|W\|_{\infty}
\]
we can continue to bound $\|[L_{s,q}^{N}W](x)\|$ up to $O(\|W\|_{\Lip}\g^{-M})$
by 
\begin{align*}
Cq^{-1/4}\exp(\k_{1})\|W\|_{\infty}\sum_{\a^{M}}\exp(\tau_{a}^{M}(\a^{M}\a_{0}^{R}x)) & \leq Cq^{-1/4}\exp(\k_{1})\|W\|_{\infty}L_{a}^{N}[1](T^{M}\a^{M}\a_{0}^{R}x)\\
 & =Cq^{-1/4}\exp(\k_{1})\|W\|_{\infty}.
\end{align*}

We have now proved, by choosing $N>\k_{10}\log q$ so that there is
room for the requisite $R$ and big enough $M$ the following lemma.
\begin{lem}
\label{infinitybound} Let $(q,Q_0)=1$.
There are $a_{0},q_{0},\k_{10},\e>0$ and $\g'>1$
such that when $|a-\delta|<a_{0}$, we have 
\[
\|L_{s,q}^{N}W\|_{\infty}\leq q^{-\e}\|W\|_{\infty}+{\g'}^{-N}\|W\|_{\Lip}
\]
when $N>\k_{10}\log q$, $q>q_{0}$, and $W\in E_{q}$ with $\|W\|_{\Lip}<\infty$. 
\end{lem}

\subsection{Bounds for Lipschitz norms}

In order to iterate Lemma \ref{infinitybound} (this is our aim) we
also need bounds for 
\[
\|L_{s,q}^{N}W\|_{\Lip}
\]
under the same conditions as in Lemma \ref{infinitybound}. This amounts
to estimating 
\[
\sup_{I}|[L_{s,q}^{N}W]'|
\]
and so we can proceed along similar lines as before. Indeed one calculates
from \eqref{eq:LNWexplicit} that 
\begin{align*}
[L_{s,q}^{N}W]'(x) & =\sum_{\a^{N}}([\tau_{a}^{N}+ib\tau^{N}]\circ\a^{N})'(x)\exp([\tau_{a}^{N}+ib\tau^{N}](\a^{N}x))c_{q}^{N}(\a^{N}x)W(\a^{N}x)\\
 & +\sum_{\a^{N}}\exp([\tau_{a}^{N}+ib\tau^{N}](\a^{N}x))c_{q}^{N}(\a^{N}x)[W\circ\a^{N}]'(x)
\end{align*}
using that $c_{q}^{N}$ is locally constant. The second set of terms
are bounded by 
\[
\frac{D}{\g^{N}}\sum_{\a^{N}}\exp(\tau_{a}^{N}(\a^{N}x))\|W\|_{\Lip}
\]
which can be bounded by 
\[
\frac{D}{\g^{N}}\|W\|_{\Lip}L_{a}^{N}[1](x)=\frac{D}{\g^{N}}\|W\|_{\Lip}.
\]
So we have 
\begin{equation}
[L_{s,q}^{N}W]'(x)=\Sigma+O(\frac{D}{\g^{N}}\|W\|_{\Lip})\label{eq:Ssplit}
\end{equation}
where 
\[
\Sigma:=\sum_{\a^{N}}([\tau_{a}^{N}+ib\tau^{N}]\circ\a^{N})'(x)\exp([\tau_{a}^{N}+ib\tau^{N}](\a^{N}x))c_{q}^{N}(\a^{N}x)W(\a^{N}x).
\]
We can go through the same decoupling argument as before to get 

\begin{align*}
\Sigma & =\sum_{\a^{N}}([\tau_{a}^{N}+ib\tau^{N}]\circ\a^{N})'(x)\exp([\tau_{a}^{N}+ib\tau^{N}](\a^{N}x))c_{q}^{N}(\a^{N}x)W(\a^{M}x_{0}(i_0))\\
 & +O\left(\|W\|_{\Lip}\frac{D}{\g^{M}}\diam(I)\sum_{\a^{N}}|([\tau_{a}^{N}+ib\tau^{N}]\circ\a^{N})'(x)|\exp(\tau_{a}^{N}(\a^{N}x))\right),
\end{align*}
recalling $i_0=i_0(\alpha^M)$. Note that since there are constants $C_{1}$ and $a_{0}$ such that
when $|a-\delta|<a_{0}$ we have 
\[
|[\tau_{a}^{N}\circ\a^{N}]'(x)|\leq C_{1}
\]
for $x\in I$ (see for example \cite[pg. 138]{NAUD}), we have 

\begin{equation}
|[[\tau_{a}^{N}+ib\tau^{N}]\circ\a^{N}]'(x)|\leq C_{1}+|b|\sup_{I}|\tau'|\sum_{i=0}^{N-1}\frac{D}{\g^{i}}\leq\k_{11}\label{eq:k11}
\end{equation}

for some $\k_{11}=\k_{11}(a_{0},b_{0})$ when $|b|\leq b_{0}$. Therefore
we have the decoupled equation 

\[
\Sigma=\sum_{\a^{N}}([\tau_{a}^{N}+ib\tau^{N}]\circ\a^{N})'(x)\exp([\tau_{a}^{N}+ib\tau^{N}](\a^{N}x))c_{q}^{N}(\a^{N}x)W(\a^{M}x_0(i_0))+O_{b_{0}}(\|W\|_{\Lip}\g^{-M})
\]
valid when $|b|<b_{0}$ and $|a-\delta|<a_{0}$ for some fixed $a_{0}$.
We denote the first of these two terms by $\Sigma'$. Now similarly
to before we define complex valued measures

\begin{equation}
\mu'_{s,x,\a^{M}}=\sum_{\alpha^{R}}^*([\tau_{a}^{N}+ib\tau^{N}]\circ\a^{M}\alpha^{R})'(x)\exp([\tau_{a}^{N}+ib\tau^{N}](\alpha^{M}\alpha^{R}x))\delta_{c_{q}^{R}(\a^{R}x)^{-1}}\label{eq:mudashdef}
\end{equation}
\[
\varphi_{f,\a^{M}}=\sum_{g\in\G_{q}}f(\a^{M}x_{0}(i_0))\lvert_{g}\:\delta_{gc_{q}^{M}(\a^{M}x_{0}(i_0))^{-1}}
\]
for $f\in C^{1}(I;\G^{q})$, $\a^{M}$ as in \eqref{eq:splitting}.
Then the key observation is that 
\begin{equation}
\|\Sigma'\|=\left\Vert \sum_{\a^{M}}\varphi_{W,\a^{M}}\star\mu'_{s,x,\a^{M}}\right\Vert _{l^{2}(\G_{q})}.\label{eq:keyS}
\end{equation}

\subsection{Bounds for $\mu'_{s,x,\protect\a^{M}}$}

We have 
\[
\|\mu'_{s,x,\a^{M}}\|_{1}\leq\sup_{I}|[\tau_{a}^{N}+ib\tau^{N}]\circ\a^{M}\alpha^{R})'(x)|\sum_{\alpha^{R}}^*\exp(\tau_{a}^{N}(\alpha^{M}\a^{R}x)).
\]
By equation \eqref{eq:k11}, $\mu'_{s,x,\a^{M}}$ is dominated (in
absolute value) by $\k_{11}(b_{0})\mu{}_{s,x,\a^{M}}$ when $|b|<b_{0}$.
Thus we can use our previous bound \eqref{eq:flattening} to deduce
that for the same choice of $R=c\log q$ and $a$ as before, 

\begin{align*}
\|\Sigma'\| & \leq Cq^{-1/4}\exp(\k_{1})\k_{11}\|W\|_{\infty}\sum_{\a^{M}}\exp(\tau_{a}^{M}(\alpha^{M}\a_{0}^{R}x))\\
 & \leq Cq^{-1/4}\exp(\k_{1})\k_{11}\|W\|_{\infty}L_{a}^{N}[1](\a_{0}^{R}x)\leq Cq^{-1/4}\exp(\k_{1})\k_{11}\|W\|_{\infty}
\end{align*}
whenever $|a-\delta|<a_{0}$, $|b|<b_{0}$ are the ranges specified
by previous Lemmas and $N>\k_{12}\log q$. It now follows from \eqref{eq:Ssplit}
that with these conditions on $N,q,a,b$ we have in light of Lemma
\ref{infinitybound} 
and the prior bound \eqref{eq:Ssplit}
\begin{equation}
\|L_{s,q}^{N}W\|_{\Lip}\leq\k_{13}q^{-\e}\|W\|_{\infty}+\k_{14}\g^{-N}\|W\|_{\Lip}+\g'^{-N}\|W\|_{\Lip}\label{eq:preiterated}
\end{equation}
for some $\e>0$ when $W\in E_{q}$.

By iterating the estimate \eqref{eq:preiterated} one obtains
\begin{lem}
\label{powerssmallfreq}For $b_{0}>0$ given, there are $a_{0}$,
$q_{0}$, $\k$ and $\e>0$ such that when $|a-\delta|<a_{0}$, $|b|<b_{0}$
and $N=\lceil\k\log q\rceil$ with $q>q_{0}$ and  $(q,Q_0)=1$ we have 
\[
\|L_{s,q}^{nN}W\|_{\Lip}\leq q^{-n\e}
\]
for all $E_{q}$-valued $W\in C^{1}(I;\C^{\G_{q}})$ with $\|W\|_{\Lip}=1$. 
\end{lem}

\subsection{The new subspace structure and the proof of Part \ref{enu:mainmodular}
of Theorem \ref{thm:maintransfer}.}

We note first the following consequence of Lemma \ref{powerssmallfreq}.
\begin{lem}
\label{powerssmallfreq2} For all $b_{0}>0$ there are $0<\rho<1$,
$a_{0}$, $q_{0}$ and $C$ such that when $|a-\delta|<a_{0}$, $|b|\leq b_{0}$
and 
$q_0<q$, $(q,Q_0)=1$,
 we have for all $m>0$ 
\[
\|L_{s,q}^{m}f\|_{C^{1}}\leq Cq^{C}\rho^{m}\|f\|_{C^{1}}\quad\text{when \ensuremath{f\in E_{q}}.}
\]

\end{lem}
This is an easy exercise and the reader can get the details from the
proof of \cite[Theorem 4.3]{OW}. 

Recall the new subspace structure of $\G_{q}$. For any $q'|q$ there
is a projection $\G_{q}\to\G_{q'}$. The kernel of this projection
will be denoted $\G_{q}(q')$, the congruence subgroup of level $q'$
in $\G_{q}$. These have the property that if $q''|q'$ then $\G_{q}(q')\leq\G_{q}(q'')$.
This groups give an orthogonal decomposition of the right regular
representation 
\begin{equation}
\C^{\G_{q}}=\bigoplus_{q'|q}E_{q'}^{q}\label{eq:decompregular}
\end{equation}
where $E_{q'}^{q}$ consists of functions invariant under $\G_{q}(q')$
but not invariant under $\G_{q}(q'')$ for any $q''|q'$, $q''\neq q'$.
Then the $E_{q}$ from before matches $E_{q}^{q}$ as defined here.

The decomposition \eqref{eq:decompregular} gives rise to a corresponding
direct sum decomposition 
\[
C^{1}(I;\C^{\G_{q}})=C^{1}(I)\oplus\bigoplus_{1\neq q'|q}C^{1}(I;E_{q'}^{q}).
\]
It is clear that the subspaces $E_{q'}^{q}$ are invariant under the
transfer operator $L_{s,q}$ and taking derivatives.

Also note that if $f\in E_{q'}^{q}$ then $f$ descends to a well
defined function $F$ on $\G_{q}/\G_{q}(q')\cong\G_{q'}$ which is
not invariant under any congruence subgroup of $\G_{q'}$, hence in
$E_{q'}^{q'}$. Also, if $G$ is a function in $E_{q'}^{q'}$ then
$G$ lifts through the previous isomorphism to a function $g$ in
$E_{q'}^{q}$ for any $q'|q$. This gives rise to a map of Banach
spaces 
\[
\Phi_{q,q'}:C^{1}(I;E_{q'}^{q'})\to C^{1}(I;E_{q'}^{q})
\]
for any $q'|q$ with the property that 
\[
\|\Phi_{q,q'}(f)\|_{C^{1}}=\sqrt{|\G_{q}(q')|}\|f\|_{C^{1}}.
\]

This map is equivariant under the transfer operators in the sense
that 
\[
\Phi_{q,q'}[L_{s,q'}f]=L_{s,q}\Phi_{q,q'}[f]
\]
for any $f\in E_{q'}^{q'}$. In other words, the action of $L_{s,q}$
on a summand in \eqref{eq:decompregular} is determined by the action
of the corresponding transfer operator on $E_{q'}^{q'}$ for some
$q'|q$. We decompose $f\in C^{1}(I;\C^{\G_{q}})$ as 
\[
f=f_{1}+\sum_{1\neq q'|q}f_{q'}
\]
with $f_{q'}\in E_{q'}^{q}$. 
If we assume that
$q$ has no proper divisors $\leq q_{0}$ from Lemma \ref{powerssmallfreq2},
then for any $m$, with all norms $C^{1}$ norms, 
\begin{align*}
\|L_{s,q}^{m}f-L_{s,q}^{m}f_{1}\| & \leq\sum_{q_{0}<q'|q}\|L_{s,q}^{m}f_{q'}\|\\
 & =\sum_{q_{0}<q'|q}\sqrt{\# \G_{q}(q')}\|L_{s,q'}^{m}\Phi_{q,q'}^{-1}f_{q'}\|\\
 & \leq C\sum_{q_{0}<q'|q}\sqrt{\# \G_{q}(q')}(q')^{C}\rho^{m}\|\Phi_{q,q'}^{-1}f_{q'}\|\\
 & \leq Cq^{C}\rho^{m}\sum_{1\neq q'|q}\|f_{q'}\|.
\end{align*}
This bound can be changed to 
\[
\|L_{s,q}^{m}f-L_{s,q}^{m}f_{1}\|\leq C'q^{C'}\rho^{m}\|f\|
\]

for some $C'=C'(\G,b_{0})$
by noting that individually 
\[
\|f_{q'}\|\leq\|f\|
\]
and that any number $q$ has $\ll_{\e}q^{\e}$ divisors for any $\e>0$.
The analogous estimates hold for the unnormalized $\L_{s,q}$ (by
perturbation theory and \eqref{eq:perturbed}). That is, by possibly
adjusting constants slightly and decreasing $a_{0}$ 
\[
\|L_{s,q}^{m}f-L_{s,q}^{m}f_{1}\|\leq C'q^{C'}\rho^{m}\|f\|.
\]
In particular Part \ref{enu:mainmodular} of Theorem \ref{thm:maintransfer}
now follows from the case that $f_{1}=0$ so that $f\in C^{1}(I;\C^{\G_{q}}\ominus1)$.

\newpage
\appendix

\newcommand{\norm}[1]{\left\|#1\right\|}

\newcommand\pf{\begin{proof}}
\newcommand\epf{\end{proof}}

\newcommand\bp{\begin{pmatrix}}
\newcommand\ep{\end{pmatrix}}
\newcommand\ben{\begin{enumerate}}
\newcommand\een{\end{enumerate}}
\newcommand\be{\begin{equation}}
\newcommand\ee{\end{equation}}
\newcommand\benn{\begin{equation*}}
\newcommand\eenn{\end{equation*}}
\newcommand\bea{\begin{eqnarray}}
\newcommand\eea{\end{eqnarray}}
\newcommand\beann{\begin{eqnarray*}}
\newcommand\eeann{\end{eqnarray*}}

\newcommand{\twocase}[5]{#1 \begin{cases} #2 & \text{#3}\\ #4
&\text{#5} \end{cases}   }

\newcommand{\threecase}[6]{
\begin{cases} 
#1 & \text{#2}\\ 
#3 & \text{#4}\\
#5 & \text{#6}\\ 
\end{cases}   }

\newcommand{\fourcase}[8]{
\begin{cases} 
#1 & \text{#2}\\ 
#3 & \text{#4}\\
#5 & \text{#6}\\ 
#7 & \text{#8}\\ 
\end{cases}   }

\newcommand{\bo}{{\bf1}}

\newcommand{\bE}{\mathbb{E}}

\newcommand{\Ep}[2]{\underset{#1}{\mathbb{E}'}\left[#2\right]}
\newcommand{\bA}{\mathbb{A}}
\newcommand{\bP}{\mathbb{P}}

\newcommand{\z}{\mathbb{Z}}

\newcommand{\q}{\mathbb{Q}}

\newcommand{\bH}{\mathbb{H}}
\newcommand{\bD}{\mathbb{D}}
\newcommand{\boldH}{\mathbb{H}}

\newcommand{\bb}{{\bf b}}
\newcommand{\bba}{{\bf a}}
\newcommand{\bbA}{{\bf A}}
\newcommand{\bbe}{{\bf e}}
\newcommand{\bi}{{\bf i}}
\newcommand{\bbk}{{\bf k}}
\newcommand{\bj}{{\bf j}}
\newcommand{\bm}{{\bf m}}
\newcommand{\bs}{{\bf s}}
\newcommand{\bx}{{\bf x}}
\newcommand{\by}{{\bf y}}
\newcommand{\bu}{{\bf u}}
\newcommand{\bv}{{\bf v}}
\newcommand{\bw}{{\bf w}}
\newcommand{\dz}{d\mu(z)}
\newcommand{\vbZ}{\overset\rightarrow{\bf Z}}
\newcommand{\bG}{{\bf G}}
\newcommand{\bbG}{{\mathbb G}}

\newcommand{\fa}{\frak{a}}
\newcommand{\fb}{\frak{b}}
\newcommand{\fc}{\frak{c}}
\newcommand{\fd}{\frak{d}}
\newcommand{\fe}{\frak{e}}
\newcommand{\ff}{\frak{f}}
\newcommand{\fh}{\frak{h}}
\newcommand{\fj}{\frak{j}}
\newcommand{\fk}{\frak{k}}
\newcommand{\fl}{\frak{l}}
\newcommand{\fm}{\frak{m}}
\newcommand{\fn}{\frak{n}}
\newcommand{\fo}{\frak{o}}
\newcommand{\fp}{\frak{p}}
\newcommand{\fq}{\frak{q}}
\newcommand{\fr}{\frak{r}}
\newcommand{\fs}{\frak{s}}
\newcommand{\ft}{\frak{t}}
\newcommand{\fu}{\frak{u}}
\newcommand{\fv}{\frak{v}}
\newcommand{\fw}{\frak{w}}
\newcommand{\fx}{\frak{x}}
\newcommand{\fy}{\frak{y}}
\newcommand{\fz}{\frak{z}}

\newcommand{\fA}{\frak{A}}
\newcommand{\fB}{\frak{B}}
\newcommand{\fC}{\frak{C}}
\newcommand{\fD}{\frak{D}}
\newcommand{\fE}{\frak{E}}
\newcommand{\fF}{\frak{F}} 
\newcommand{\fG}{\frak{G}} 
\newcommand{\fH}{\frak{H}}
\newcommand{\fI}{\frak{I}}
\newcommand{\fJ}{\frak{J}}
\newcommand{\fK}{\frak{K}}
\newcommand{\fL}{\frak{L}}
\newcommand{\fM}{\frak{M}}
\newcommand{\fN}{\frak{N}}
\newcommand{\fO}{\frak{O}}
\newcommand{\fP}{\frak{P}}
\newcommand{\fQ}{\frak{Q}}
\newcommand{\fR}{\frak{R}}
\newcommand{\fS}{\frak{S}}
\newcommand{\fT}{\frak{T}}
\newcommand{\fU}{\frak{U}}
\newcommand{\fV}{\frak{V}}
\newcommand{\fW}{\frak{W}}
\newcommand{\fX}{\frak{X}}
\newcommand{\fY}{\frak{Y}}
\newcommand{\fZ}{\frak{Z}}

\newcommand{\sA}{\mathscr{A}}
\newcommand{\sB}{\mathscr{B}}
\newcommand{\sC}{\mathscr{C}}
\newcommand{\sD}{\mathscr{D}}
\newcommand{\sE}{\mathscr{E}}
\newcommand{\sF}{\mathscr{F}}
\newcommand{\sG}{\mathscr{G}}
\newcommand{\sH}{\mathscr{H}}
\newcommand{\sI}{\mathscr{I}}
\newcommand{\sJ}{\mathscr{J}}
\newcommand{\sK}{\mathscr{K}}
\newcommand{\sL}{\mathscr{L}}
\newcommand{\sM}{\mathscr{M}}
\newcommand{\sN}{\mathscr{N}}
\newcommand{\sO}{\mathscr{O}}
\newcommand{\sP}{\mathscr{P}}
\newcommand{\sQ}{\mathscr{Q}}
\newcommand{\sR}{\mathscr{R}}
\newcommand{\sS}{\mathscr{S}}
\newcommand{\sT}{\mathscr{T}}
\newcommand{\sU}{\mathscr{U}}
\newcommand{\sV}{\mathscr{V}}
\newcommand{\sW}{\mathscr{W}}
\newcommand{\sX}{\mathscr{X}}
\newcommand{\sY}{\mathscr{Y}}
\newcommand{\sZ}{\mathscr{Z}}

\newcommand{\cA}{\mathcal{A}}
\newcommand{\cB}{\mathcal{B}}
\newcommand{\cC}{\mathcal{C}}
\newcommand{\cD}{\mathcal{D}}
\newcommand{\cE}{\mathcal{E}}
\newcommand{\cF}{\mathcal{F}} 
\newcommand{\cG}{\mathcal{G}} 
\newcommand{\cH}{\mathcal{H}}
\newcommand{\cI}{\mathcal{I}}
\newcommand{\cJ}{\mathcal{J}}
\newcommand{\cK}{\mathcal{K}}
\newcommand{\cL}{\mathcal{L}}
\newcommand{\cM}{\mathcal{M}}
\newcommand{\cN}{\mathcal{N}}
\newcommand{\cO}{\mathcal{O}}
\newcommand{\cP}{\mathcal{P}}
\newcommand{\cp}{\mathcal{p}}
\newcommand{\cQ}{\mathcal{Q}}
\newcommand{\cR}{\mathcal{R}}
\newcommand{\cS}{\mathcal{S}}
\newcommand{\cT}{\mathcal{T}}
\newcommand{\cU}{\mathcal{U}}
\newcommand{\cV}{\mathcal{V}}
\newcommand{\cW}{\mathcal{W}}
\newcommand{\cX}{\mathcal{X}}
\newcommand{\cY}{\mathcal{Y}}
\newcommand{\cZ}{\mathcal{Z}}

\newcommand{\Or}{\ensuremath{{\mathcal O}}}
\newcommand{\n}{\N}

\newcommand{\ha}{\aleph}     
\newcommand{\hb}{\beth}     
\newcommand{\hg}{\gimel}     
\newcommand{\hd}{\daleth}     

\newcommand{\ga}{\alpha}     
\newcommand{\gb}{\beta}      
\newcommand{\gd}{\delta}     
\newcommand{\gD}{\Delta}     
\newcommand{\sq}{\square}    
\newcommand{\gn}{\eta}       
\newcommand{\gep}{\epsilon}  
\newcommand{\vep}{\varepsilon} 

\newcommand{\Gi}{\Gamma_{\infty}}      

\newcommand{\gL}{\Lambda}    
\newcommand{\gi}{\iota}    
\newcommand{\gl}{\lambda}    
\newcommand{\gk}{\kappa}    
\newcommand{\gs}{\sigma}     
\newcommand{\gS}{\Sigma}     
\newcommand{\gt}{\theta}     
\newcommand{\gT}{\Theta}     
\newcommand{\gz}{\zeta}      
\newcommand{\gU}{\Upsilon}      
\newcommand{\gu}{\upsilon}      
\newcommand{\gw}{\omega}      
\newcommand{\gW}{\Omega}      

\newcommand{\vf}{\varphi}      
\newcommand{\vt}{\vartheta}      

\newcommand{\foh}{\frac{1}{2}}  

\newcommand{\vectwo}[2]
{\left(\begin{array}{c}
                        #1    \\
                        #2
                          \end{array}\right) }

\newcommand{\vecthree}[3]
{\left(\begin{array}{c}
                        #1    \\
                        #2 \\
                        #3
                          \end{array}\right) }

\newcommand{\mattwos}[4]
{\bigl( \begin{smallmatrix}
                        #1  & #2   \\
                        #3 &  #4
\end{smallmatrix} \bigr)
}

\newcommand{\mattwo}[4]
{\left(\begin{array}{cc}
                        #1  & #2   \\
                        #3 &  #4
                          \end{array}\right) }

\newcommand{\matthree}[9]
{\left(\begin{array}{ccc}
                        #1  & #2 & #3  \\
                        #4  & #5 & #6  \\
                        #7 &  #8 & #9
                          \end{array}\right) }

\newcommand{\Res}{\operatorname{Res}}

\newcommand{\Stab}{\operatorname{Stab}}

\newcommand{\Spec}{\operatorname{Spec}}
\newcommand{\New}{\operatorname{New}}

\newcommand{\sech}{\operatorname{sech}}
\newcommand{\csch}{\operatorname{csch}}

\newcommand{\diag}{\operatorname{diag}}

\newcommand{\Isom}{\operatorname{Isom}}

\newcommand{\Div}{\operatorname{div}}
\newcommand{\grad}{\operatorname{grad}}

\newcommand{\sqf}{\operatorname{sqf}}

\newcommand{\ord}{\operatorname{ord}}

\newcommand{\Zcl}{\operatorname{Zcl}}

\newcommand{\Inn}{\operatorname{Inn}}
\newcommand{\Epi}{\operatorname{Epi}}
\newcommand{\Aut}{\operatorname{Aut}}
\newcommand{\Cay}{\operatorname{Cay}}

\newcommand{\PGL}{\operatorname{PGL}}
\newcommand{\PSL}{\operatorname{PSL}}

\newcommand{\Sp}{\operatorname{Sp}}
\newcommand{\SU}{\operatorname{SU}}

\newcommand{\ad}{\operatorname{ad}}
\newcommand{\Ad}{\operatorname{Ad}}

\newcommand{\Id}{\operatorname{Id}}

\newcommand{\sgn}{\operatorname{sgn}}
\newcommand{\supp}{\operatorname{supp}}
\newcommand{\sym}{\operatorname{sym}}

\renewcommand{\mod}{\operatorname{mod}}
\newcommand{\smod}[1]{(\operatorname{mod} #1)}
\newcommand{\Tr}{\operatorname{Tr}}
\newcommand{\Det}{\operatorname{Det}}
\newcommand{\discr}{\operatorname{discr}}
\renewcommand{\hat}{\widehat} 

\newcommand{\Proj}{\operatorname{Proj}}

\newcommand{\Disc}{\operatorname{Disc}}
\newcommand{\Area}{\operatorname{Area}}

\renewcommand{\Re}{{\mathfrak{Re}}}
\renewcommand{\Im}{{\mathfrak{Im}}}

\newcommand{\<}{\left\langle}
\renewcommand{\>}{\right\rangle}
\newcommand{\la}{\langle}
\newcommand{\ra}{\rangle}
\newcommand{\bk}{\backslash}
\newcommand{\ba}{\backslash}
\newcommand{\GbkH}{\G\bk\boldH}
\newcommand{\GibkH}{\Gi\bk\boldH}
\newcommand{\GqbkH}{\G_0(q)\bk\boldH}

\newcommand{\dd}{\partial}

\newcommand{\ddxx}{{\partial^2\over\partial x^2}}
\newcommand{\ddyy}{{\partial^2\over\partial y^2}}
\newcommand{\ddx}{{\partial\over\partial x}}
\newcommand{\ddy}{{\partial\over\partial y}}

\newcommand{\op}{\operatorname}

\newcommand{\Li}{\operatorname{Li}}
\newcommand{\iL}{\operatorname{iL}}

\newcommand{\lcm}{\operatorname{lcm}}
\def\mult{\mathrm{mult}}

\section{Thermodynamic expansion to arbitrary moduli\\ By Jean Bourgain, Alex Kontorovich and Michael Magee}

\subsection{Statements}
We import all the notation from the rest of the paper. 
We are led to study the measure $\mu$ on $G=\SL_{2}(q)$ given by

\be\label{eq:muIs}
\mu=\sum^*_{\alpha^{R}}\exp([\tau_{a}^{N}+ib\tau^{N}](\a^{M}\alpha^{R}x))\delta_{c_{q}^{R}(\a^{R}x)},
\ee
this differs from the $\mu_{s,x,\a^{M}}$ of equation \eqref{eq:mudefinition} by taking inverses of group elements. This makes spectral bounds for the right action of $\mu_{s,x,\a^{M}}$ and those for the left action of $\mu$ equivalent. 
Here $N=M+R$, $x\in I$, 
\begin{equation}
\a^M=g_{i_1}g_{i_2}\ldots g_{i_M}
\end{equation}
is fixed, and the starred summation means that it is restricted to those 
\begin{equation}
\alpha^R=g_{i_{M+1}}g_{i_{M+2}}\ldots g_{i_{N}}
\end{equation} where the sequence $g_{i_1},\ldots,g_{i_{N}}$ is admissible and 
$\alpha^R$ is a well defined local branch of $T^{-R}$ near $x$. In practice this may rule out one possible value for $i_N$. See Section \ref{expingredients} for more details. Also recall the ``new subspace" $E_q\subset l^2(G)$ defined in Section \ref{expingredients} and the constant $a_0$ coming from Proposition \ref{blackbox}.

Our goal in this Appendix is to prove the following
\begin{thm}\label{thm:main} There is a finite modulus $Q_0$ and $c >0$ such that when $R\approx c \log q $, $(q,Q_0)=1$, $|a-\delta|<a_{0}$ and $\vf\in E_{q}$,  we have
\be\label{eq:main}
\|\mu * \vf\|_{2} \ \le \ C\
q^{-1/4}\,
B\,
\|\vf\|_{2},
\ee
given that
$$
\|\mu\|_{1}\ <\ B.
$$
\end{thm}

Recall that in Section \ref{expingredients} we chose for each $\alpha^M$ an $i_0=i_0(\alpha^M)$ such that $\alpha^M$ is a well-defined local 
branch of $T^{-M}$ on $I_{i_0}$. We also chose for each $i$ an $x(i)$ in $I_i$.  More generally, for each admissible composition $\alpha = g_{i_1}\ldots g_{i_j}$
of semigroup elements we now choose an $i(\alpha)$ such that $\alpha$ is a well defined branch of $T^{-j}$ on $I_{i(\alpha)}$. This choice depends only on $i_j$.  Let $o=x(i(\alpha^R))$.

To begin, we define a measure $\nu$ by
\be\label{eq:nuDef}
\nu \ \equiv \
\exp(\tau_{a}^{M}(\ga^{M}x(i_0)))\mu_{1},
\ee
where $\mu_{1}$ is the measure given by
\be\label{eq:mu1Def}
\mu_{1} \ \equiv \
\sum_{\ga^{R}}^*
\exp(\tau_{a}^{R}(\ga^{R}o))
\gd_{c_{q}^{R}(\ga^{R}o)}
.
\ee

\begin{lem}
We have
\be\label{eq:muBndNu}
|\mu| \ \le \ C\, \nu.
\ee
\end{lem}
\pf
Use the ``contraction property'' in \eqref{align:decouple} and
argue
 as in the proof of Lemma \ref{mul1}.
\epf

We will now manipulate $\mu_{1}$. We assume that $R$ can be decomposed further  as 
\be\label{eq:RpL}
R \ = \ R'L,
\ee
with $L$ to be chosen later (a sufficiently large constant independent of $R'$ and $q$). Now split $\ga^R$ as
\begin{equation}\label{eq:arsplitting}
\ga^R \  = \  \ga_{R'}^L \ga_{R' - 1}^L \ldots \ga^L_2 \ga^L_1
,
\end{equation}
where the $\ga_{k}^L$ are  branches of $T^{-L}$ given by
\begin{equation}
\alpha^L_{R'}=g_{i_{M+1}}\ldots g_{i_{M+L}},\quad\alpha^L_{R'-1}=g_{i_{M+L+1}}\ldots g_{i_{M+2L}}
\end{equation}
and so on. For each $0\leq p \leq R'-1$ we also split 
$$
\ga^L_{R'-p} \  = \  \ga^{L-2}_{R'-p} \alpha^{(2)}_{R'-p} ,
$$
where $\ga^{L-2}_{R'-p} = g_{i_{M+pL+1}}\ldots g_{i_{M+(p+1)L-2}}$  and $\alpha^{(2)}_{R'-p}=g_{i_{M+(p+1)L-1}}g_{i_{M+(p+1)L}}$. The reason for isolating two indices
will become clear later.

Write out
\bea
\nonumber
\tau_a^R(\ga^R o)  &=& \sum_{ i = 0 }^{R-1} \tau_a( T^i \ga^R o ) 
\\
\nonumber
&=& \sum_ {i = 0 }^{R'-1} \sum_{ \ell = 0}^{L-1} \tau_a( T^{  i L+ \ell } \ga^R o) 
\\
\nonumber
&=&  \sum_ {i = 0 }^{R'-1} \sum_{ \ell = 0}^{L-1} \tau_a( T^{  i L+ \ell } \ga^L_{R' - i } \ga^L_{R' - i - 1} \ldots \ga^L_1   o ) 
\\
\label{eq:lastline}
&=& \sum_{ i = 0}^{R' - 1} \tau_a^L (  \ga^L_{R' - i } \ga^L_{R' - i - 1} \ldots \ga^L_1 (o )) . 
\eea
We now perform decoupling term by term in the above.   We will use the shorthand
$$
\ga^{ L j  }  \ \equiv \  \ga^L_{j } \ga^{L}_{j - 1} \ldots \ga^L_1 .
$$
For $j \ge2$,
we  compare  each term
 in \eqref{eq:lastline}
  of the form
$$
\tau_a^L ( \ga^{ L j }( o ) )
$$
to 
$$
 \tau_a^L (  \ga^L_{j } \ga^{L-2}_{j-1} x(i(\alpha^{L-2}_{j-1})) ) 
. 
$$
This gives
\bea
\nonumber
\tau_a^L ( \ga^{ L j }( o ) ) &=&  \tau_a^L (  \ga^L_{j } \ga^{L-2}_{j-1} x(i(\alpha^{L-2}_{j-1})) ) + O \bigg( \sup | [ \tau_a^L \circ \ga^L_{j } ]' | d( \ga^{L-2}_{j- 1} x(i(\alpha^{L-2}_{j-1})),   \ga^{L-2}_{j-1}\ga^{(2)}_{j-1} \ldots \ga^L_1 o )\bigg) 
\\
\label{eq:constant}
&=&  \tau_a^L (  \ga^L_{j } \ga^{L-2}_{j-1} x(i(\alpha^{L-2}_{j-1}))) + O ( \g^{-(L-2) } ) ,
\eea
where we used the bound \eqref{eq:perturbbound} of Proposition \ref{blackbox}, valid when $a$ is within $a_0$ of  $\delta$. 

We will also use the formula
\begin{equation}\label{eq:cqsplit}
\delta_{c_q^R(\ga^R o) }  =  \delta_{c_q^L(\ga^L
o) }  *  \delta_{c_q^L(\ga^{2L} o) } *  \delta_{c_q^L(\ga^{3L} o) } * \ldots *      \delta_{c_q^L(\ga^{R' L} o) } .
\end{equation}
Then combining \eqref{eq:lastline} and \eqref{eq:cqsplit}, we write
\bea
\nonumber
\mu_1 &=&  \sum_{ \ga^{L}_1 , \ga^{L-2}_2 , \ldots , \ga^{L-2}_{R'}    }^* \sum_{ \ga^{(2)}_2 ,\ldots \ga^{(2)}_{R'}  }^* \exp( \tau_a^R  (\ga^R o ) )) \delta_{ c_q^R(\ga^R o)} 
\\
\nonumber
&=&  \sum_{ \ga^{L}_1 , \ga^{L-2}_2 , \ldots , \ga^{L-2}_{R'}    }^* \sum_{ \ga^{(2)}_2 ,\ldots \ga^{(2)}_{R'}  }^* \exp \left ( \sum_{ j = 1}^{R'} \tau_a^L ( \ga^{ j L }(o))  \right)  
\times
\\
\label{eq:innersum}
&&
\hskip1in
\delta_{c_q^L(\ga^L o) }  *  \delta_{c_q^L(\ga^{2L} o) } *  \delta_{c_q^L(\ga^{3L} o) } * \ldots *      \delta_{c_q^L(\ga^{R' L} o) }  
.
\eea
Starred summation means that  the outer sum is restricted to be compatible with $\alpha^M$ and $x$, and given the collection of  $\alpha^{L-2}_k$ from the outer sum, we then restrict to those $\alpha^{(2)}_k$ that form admissible compositions overall.
We now decouple, replacing each term of the form
$$
e^{ \tau_a^L ( \ga^{ jL }(o) )  } \ \mapsto \ e^{ \tau_a^L (\ga^L_j   \ga^{L-2}_{j-1} x(i(\alpha^{L-2}_{j-1}))) } 
 \ \equiv \ \gb_{j}
$$
with $j \ge2$,
at a cost of a multiplicative factor of $\exp(  c \g^{-L } )$; here $c$ is proportional to the implied constant of \eqref{eq:constant}. When $j = 1$, no replacement is performed, and we set
 $\gb_{1}\equiv e^{ \tau_a^L (\ga^L_{1} o)  }$. 

Inserting this into \eqref{eq:innersum} gives
\bea
\label{eq:mu1Bnd2}
\mu_{1}
&\leq &
\sum_{ \ga^{L-2}_1 , \ga^{L-2}_2 , \ldots , \ga^{L-2}_{R'}    }^* 
\sum_{\ga_{1}^{(2)}}^*
\gb_{1}
\delta_{c_q^L(\ga^L o) }  *
\\
\nonumber
&&
\exp( c \g^{-L} )^{R'-1} 
\left(
\sum_{ \ga^{(2)}_2 ,\ldots \ga^{(2)}_{R'}  }^* \prod_{ j =2 }^{R'} 
\gb_{j}\
\delta_{c_q^L(\ga^{2L} o) } *  \delta_{c_q^L(\ga^{3L} o) } * \ldots *      \delta_{c_q^L(\ga^{R' L} o) }  
\right).
\eea
Note that, although $\gb_{j}$ depends on all of the indices in $\ga_{j}^{L}\ga_{j-1}^{L-2}$,  
because
$\ga_{j}^{L-2}$ and $\ga_{j-1}^{L-2}$ are fixed in the outermost sum, we 
treat
$\gb_{j}$ as a function of  $\ga_{j}^{(2)}$. 

We claim that each term  $c_q^L(\ga^{jL } o)$ also only depends on one $\ga_j^{(2)}$. This is because 
we have $\ga^{jL} = g_{k_1} \ldots g_{k_L} \ga^{(j-1)L}$ for some choice of $g_{k_m}$, and
hence
 for whatever $o$ is chosen, we have
$$
c_q^L ( \ga^{jL } o)\  = \ c_q( g_{k_L}  \ga^{(j-1)L}o) c_q( g_{k_{L-1}} g_{k_L}  \ga^{(j-1)L} o ) \ldots  c_q( g_{k_1} \ldots g_{k_L}  \ga^{(j-1)L}o)
,
$$
see  the Dictionary on page \pageref{table:dictionary}, Section \ref{sec:Counting}. 
From the Definition of $c_q$ we have
$$
c_q( g_{ k_{m} } o' ) \ = \ g_{k_{m} } \bmod q
$$
for any $o' \in I$ where $g_{k_m}$ is a local inverse branch of $T$ near $o'$. Thus
\be\label{eq:cqL}
c_q^L( \ga^{jL} o )\  =\  g_{k_L} \ldots g_{k_1} \bmod q.
\ee
Here 
\be\label{eq:gjL}
g_{k_{L-1}}g_{k_L} \ = \  \ga^{(2)}_j.
\ee
This means we may distribute the convolution and product over the sum,
writing \eqref{eq:mu1Bnd2} as
\bea
\nonumber
\mu_{1}
&\leq& 
\exp( c \g^{-L} )^{R'-1}
\sum_{ \ga^{L-2}_1 , \ga^{L-2}_2 , \ldots , \ga^{L-2}_{R'}    }^* 
\left(
\sum_{\ga_{1}^{(2)}}^*
\gb_{1}\
\delta_{c_q^L(\ga^L o) }  
\right)*
 \left( 
\sum_{ \ga_2^{(2)} }^* 
\gb_{2}\
 \delta_{c_q^L(\ga^{2L} o) } 
\right) * \ldots
\\
\label{eq:convolve}
&&
\hskip2in
\ldots
 * \left( \sum_{\ga_{R'}^{(2)} }^*
  \gb_{R'}\
  \delta_{c_q^L(\ga^{R'L} o) } \right).
\eea
We give each convolved term in \eqref{eq:convolve} a name, defining, for each $j\ge1$, the measure
\be\label{eq:etaDef}
\eta_{j} \ = \ \eta_{j}^{(\ga_{j}^{L-2},\ga_{j-1}^{L-2})}
\
\equiv
\
\sum_{ \ga_j^{(2)} }^*
 \gb_{j }\
   \delta_{c_q^L(\ga^{jL} o) } 
.
\ee
Note this parameterization makes sense since the admissibility of $\alpha^{(2)}_j$ depends only on  $\ga_{j}^{L-2}$ and $\ga_{j-1}^{L-2}$.
We have thus  proved the following
\begin{prop}
We have
\be
\label{eq:mu1Prop}
\mu_{1}
\ \le \
\exp( c \g^{-L} )^{R'-1} 
  \sum_{ \ga^{L-2}_1 , \ga^{L-2}_2 , \ldots , \ga^{L-2}_{R'}    }^* 
 \eta_{1}
 * 
\eta_{2}
* \ldots
 *
\eta_{R'}
  .
\ee 
\end{prop}

Next we observe that each of the measures $\eta_{j}$ is nearly flat, in that their coefficients
in \eqref{eq:etaDef}
 differ by constants:
\begin{lem}\label{lem:smallfluctutations}
There is some  $c'>0$ such that for any $L>0$, for each $j\ge1$ and any $\ga_{j}^{(2)}$ and ${\ga_{j}^{(2)}}'$, we have
\be\label{eq:lem1}
{\gb_{j}'\over \gb_{j}}
\ \le\ 
c'
.
\ee
\end{lem}
\pf
The first $L-2$ terms of $\gb_{j}$ and $\gb_{j}'$ agree, so we 
again use the ``contraction property'' 
from \eqref{align:decouple}.
\epf

Since the measures $\eta_{j}$ are nearly flat, we may now apply the expansion result in \cite{BVARBITRARY}.

\begin{thm}\label{thm:flatExpand}
Assume $L$ is sufficiently large (depending only on $\G$). Then
for $\varphi\in L^{2}_{0}(G)$, we have
\be\label{eq:prop1}
\| \eta_{j}*\varphi \| _{2}
\ \le \
(1-C_{1})\,
\|\eta_{j}\|_{1}\,
\|\varphi\|_{2},
\ee
Here $C_{1}>0$ depends on $\G$  but not on $q$.
\end{thm}

\begin{proof}[Proof of Theorem \ref{thm:flatExpand}]\
Recalling \eqref{eq:etaDef}, 
we can write
\be \label{eq:etajnorm}
\| \eta_{j}*\varphi \|^{2} _{2}\  = \ \<\widetilde A \vf,\vf\>
,
\ee
where $\widetilde A
$ acts by convolution with the measure
\begin{equation}\label{eq:weightedsum}
A \equiv \sum_{\alpha_{j}^{(2)},{\alpha_{j}^{(2)}}'}^*
\gb_{j}\,\gb_{j}'\,
\delta_{c_q^L(\ga^{jL} o) c_q^L((\ga^{jL})' o)^{-1} } .
\end{equation}
Using the notation of \eqref{eq:cqL} and \eqref{eq:gjL}, note that
$$
c_q^L(\ga^{jL} o) c_q^L((\ga^{jL})' o)^{-1} 
\ = \
\ga_{j}^{(2)}\cdot g_{k_{L-1}} \ldots g_{k_1}
({\ga_{j}^{(2)}}'\cdot g_{k_{L-1}} \ldots g_{k_1})^{-1}
\ = \
\ga_{j}^{(2)}
({\ga_{j}^{(2)}}')^{-1}.
$$

We will  now appeal to the following spectral gap modulo $q$ for the group generated by the coefficients $\ga_{j}^{(2)}
({\ga_{j}^{(2)}}')^{-1}$.

\begin{prop}[Spectral gap]\label{prop:gap}
There is some modulus $Q_0$ and some $\epsilon > 0$ such that for all indices $j$, for all $q$ coprime to $Q_0$ and for all $\phi\in \ell^2_0(G)$ with $\|\phi\|_2 = 1$ there is some pair  $\ga_{j}^{(2)},{\ga_{j}^{(2)}}'$ such that 
\begin{equation}
\| \ga_{j}^{(2)}
({\ga_{j}^{(2)}}')^{-1} * \phi - \phi \|_2 > \epsilon. \label{eq:noalmostinvariant}
\end{equation}
\end{prop}

The statement of Proposition \ref{prop:gap} is well known to be equivalent to other uniform spectral gap properties.  The uniform spectral gap is known to exist in the current setting for the following reasons.

\textbf{Continued fractions setting.}
 Here we need the products  $\ga_{j}^{(2)}
({\ga_{j}^{(2)}}')^{-1}$ to generate a group with Zariski closure $\SL_{2}$. Since all sequences of $g_{i_j}$ are admissible, the $\alpha^{(2)}_j$ appearing in \eqref{eq:etaDef} do not depend on $j$.
Recall that in the continued fractions setting,
each $g_{i}$ is already a product of two generators $\mattwos011a\mattwos011b$. It is easy to see then that the $\a^{(2)}_j$ generate a Zariski dense subgroup whenever the alphabet $\A$ of $\G_\A$ has at least two letters, in fact, it would have been enough to take for the $\alpha^{(2)}$ blocks of length 1. On the other hand, we do need sufficiently many of the $\mattwos011a$ to be involved as  the products $\mattwos011a\mattwos011b^{-1}=\mattwos10{a-b}1$ are lower-triangular. 
Proposition \ref{prop:gap} then follows from the expansion result of Bourgain and Varj\'{u} \cite{BVARBITRARY}. In the cases that the  $ \ga_{j}^{(2)}
({\ga_{j}^{(2)}}')^{-1} $ generate all of $\SL_2(\Z)$, Proposition \ref{prop:gap} is a well known consequence of Selberg's ``3/16 Theorem" from \cite{SELBERGCOEF}.

 \textbf{Schottky semigroup/group setting.}
 
Note that this setting contains the case that $\G$ is a Schottky \emph{group} as in \cite{BGS2}. Again, it will be enough to show that the $\ga_{j}^{(2)}
({\ga_{j}^{(2)}}')^{-1}$ generate a Zariski dense sub\emph{group} of $\SL_2(\Z)$. This is the reason why we needed to make $\alpha_j^{(2)}$ a block of length $2$. Indeed, suppose that the Schottky semigroup is generated by at least two Schottky generators and let $g,h$ be two of these generators. For example, if  $\ga_{j}^{L-2}$ ends in $g$ while $\ga_{j-1}^{L-2}$ starts with $g^{-1}$ then the summation in \eqref{eq:etaDef} contains $\alpha_j^{(2)}$ of the form $
gh,
gh^{-1},
hg^{-1},
hh,
h^{-1}g^{-1},
h^{-1}h^{-1}
$.
It is then easy to see that the $\ga_{j}^{(2)}
({\ga_{j}^{(2)}}')^{-1}$ generate a Zariski dense group
(if 
$\G$ has more than two generators, this
 is 
 also clear). We may then apply the Bourgain-Varj\'{u} expansion result \cite{BVARBITRARY} to obtain a spectral gap for the group generated 
 by $\ga_{j}^{(2)}
({\ga_{j}^{(2)}}')^{-1}$.   Now, this group and its generator set (and hence also its expansion constant $\epsilon$ as in \eqref{eq:noalmostinvariant}) depend on $\ga_j^{L-2}$ and $\ga_{j-1}^{L-2}$ (or rather just their starting/ending letters). But as $\G$ is finitely generated, only a finite number of groups/generators arise in this way, and
we simply
take $\epsilon$ to be the worst 
one, yielding Proposition \ref{prop:gap}.

We now resume our proof of Theorem \ref{thm:flatExpand}.  Assume without loss of generality that $\|\vf\|_2=1$ and let $\ga_{j}^{(2)},{\ga_{j}^{(2)}}'$ be the pair provided by Proposition \ref{prop:gap} applied to $\vf$, and $\epsilon$ the provided constant. Since there is a uniform bound on the size of the support of $ A$, Lemma \ref{lem:smallfluctutations} gives
\begin{equation}\label{eq:coefflower}
\beta_j\beta'_j \gg \| A \|_1 
\end{equation}
with an uniform positive implied constant (here $\beta_j \beta_j'$ is the coefficient of $\alpha_j^{(2)}( \alpha_j^{(2)'})^{-1}$ in $A$).
It follows by routine arguments from \eqref{eq:coefflower} together with \eqref{eq:noalmostinvariant} for $\vf$, with the associated $\epsilon$, that the operator norm of $\widetilde A$ acting on $\ell^2_0(G)$ is
$$\| \widetilde A \|_{op}\leq (1-\epsilon')\|A\|_1$$
for some $\epsilon'$ depending on $\epsilon$. The resulting bound on \eqref{eq:etajnorm}  establishes Theorem \ref{thm:flatExpand}, since $\|A\|_1 = \|\eta_j \|^2_1$.
\end{proof}
\begin{coro}
Assume that $L$ is sufficiently large (depending only on $\G$). Then there is some $C_{2}>0$ also depending only on $\G$ so that, for any $\vf\in L^{2}_{0}(G)$, we have
\be\label{eq:mu1Bnd}
\|\mu_{1}*\vf\|_{2} \ \le \ 
(1-C_{2})^{R} \ \|\mu_{1}\|_{1} \ \|\vf\|_{2}.
\ee
\end{coro}
\pf
Beginning with \eqref{eq:mu1Prop}, apply \eqref{eq:prop1} $R'$ times to get
$$
\|\mu_{1}*\vf\|_{2} \ \le \ 
\exp(c\g^{-L})^{R'-1}
\sum_{\ga_{1}^{L-1},\dots,\ga_{R'}^{L-1}}^*
(1-C_{1})^{R'}
\prod_{j=1}^{R'}
\|\eta_{j}\|_{1}\|\vf\|_{2}.
$$
Applying contraction yet again gives
$$
\sum_{\ga_{1}^{L-1},\dots,\ga_{R'}^{L-1}}^*
\prod_{j=1}^{R'}
\|\eta_{j}\|_{1}
\ \le \
\exp(c\g^{-L})^{R'-1}
\|\mu_{1}\|_{1},
$$
whence \eqref{eq:mu1Bnd} follows on taking $L$ large enough and recalling  \eqref{eq:RpL}.
\epf

Returning to the measure $\nu$ in \eqref{eq:nuDef},
we have from \eqref{eq:mu1Bnd} that
\be\label{eq:nuBndR}
\|\nu*\vf\|_{2} \ \le \ 
(1-C_{2})^{R} \ \|\nu\|_{1} \ \|\vf\|_{2}.
\ee
To conclude Theorem \ref{thm:main}, we need the following
\begin{lem}
Let $\mu$ be a complex distribution on $G=\SL_{2}(q)$ and assume that $|\mu|\le C\nu$. Let $E_{q}\subset L_{0}^{2}(G)$ be the subspace defined in Section \ref{expingredients}, and let $A:E_{q}\to E_{q}$ be the operator acting by convolution with $\mu$. Then
\be\label{eq:1p5}
\| A\| \ \le\ C' \left[{|G| \ \|\widetilde\nu * \nu\|^{2}_{2}\over q}\right]^{1/4}.
\ee
Here $\widetilde \mu(g)=\overline{\mu(g^{-1})}$.
\end{lem}
\pf
Note that the operator $A^{*}A$ is self-adjoint, positive, and acts by convolution with $\widetilde\mu*\mu$. Let $\gl$ be an eigenvalue of $A^{*}A$. Since $A$ acts on $E_{q}$, Frobenius gives that $\gl$ has multiplicity $\mult(\gl)$ at least $Cq$. We then have that
\beann
\gl^{2}\ \mult(\gl) 
&\le& 
\tr[(A^{*}A)^{2}]
\ = \
\sum_{g\in G}
\<(A^{*}A)^{2}\gd_{g},\gd_{g}\>
\ = \
\sum_{g\in G}
\|\widetilde \mu * \mu *\gd_{g}\|_{2}^{2}
\\
& =&
|G|\
\|\widetilde \mu * \mu \|_{2}^{2}
\ \le \
C^{4}\ 
|G|\
\|\widetilde \nu * \nu \|_{2}^{2}.
\eeann
The claim follows, as $\|A\|=\max_{\gl}\gl^{1/2}$.
\epf

We apply the lemma to $\mu$ in \eqref{eq:muIs}
using
\eqref{eq:muBndNu}, giving
\be\label{eq:muConvBnd}
\|\mu * \vf\|_{2}\ \le\ C\,q^{1/2}
 \|\widetilde\nu * \nu\|^{1/2}_{2}
 .
\ee
It remains to estimate the $\nu$ convolution.
\begin{prop}
Choosing $R$ to be of size $C\log q$ for suitable $C$, we have that
\be\label{eq:nuBnd}
 \|\widetilde\nu * \nu\|_{2}
 \
 \le
 \
 2{\|\nu\|_{1}^{2}\over |G|^{1/2}}.
\ee
\end{prop}
\pf
Let
$$
\psi \ \equiv \ \gd_{e} - \frac1{|G|}\bo_{G} \ \in \ L_{0}^{2}(G),
$$
and note that 
$
\|\psi\|_{2}< 1.
$ 
Then
\beann
 \|\widetilde\nu * \nu\|_{2}
 &=&
  \|\widetilde\nu * \nu*\gd_{e}\|_{2}
\ \le\
\|\widetilde\nu * \nu*\left(\frac1{|G|}\bo_{G}\right)\|_{2}
+
\|\widetilde\nu * \nu*\psi\|_{2}
\\
& \le&
{\|\nu\|_{1}^{2}\over |G|^{1/2}}
+
\|\nu\|_{1}
\| \nu*\psi\|_{2}
,
\eeann
where we used the triangle inequality and Cauchy-Schwarz. 
Since $\psi\in L_{0}^{2}(G)$, we apply \eqref{eq:nuBndR}, giving
$$
\| \nu*\psi\|_{2}
<
(1-C_{2})^{R} \ \|\nu\|_{1}
<
{\|\nu\|_{1}\over |G|^{1/2}}
$$
by a suitable choice of $R=C\log q$. The claim follows immediately.
\epf

Finally, we give a
\pf[Proof of Theorem \ref{thm:main}]
Insert \eqref{eq:nuBnd} into \eqref{eq:muConvBnd} and use \eqref{eq:muBndNu} and $|G|>Cq^{3}$. Clearly \eqref{eq:main} holds with $B=C\|\nu\|_{1}$. 
\epf


\bibliographystyle{plain}

\begin{thebibliography}{10}

\bibitem{BORTHWICK}
David Borthwick.
\newblock {\em Spectral theory of infinite-area hyperbolic surfaces}, volume
  256 of {\em Progress in Mathematics}.
\newblock Birkh\"auser Boston, Inc., Boston, MA, 2007.

\bibitem{BOURGAINPARTIAL}
Jean Bourgain.
\newblock Partial quotients and representation of rational numbers.
\newblock {\em C. R. Math. Acad. Sci. Paris}, 350(15-16):727--730, 2012.

\bibitem{BOURGAINSURVEY}
Jean Bourgain.
\newblock Some {D}iophantine applications of the theory of group expansion.
\newblock In {\em Thin groups and superstrong approximation}, volume~61 of {\em
  Math. Sci. Res. Inst. Publ.}, pages 1--22. Cambridge Univ. Press, Cambridge,
  2014.

\bibitem{BOURGAINGAMBURDEXP}
Jean Bourgain and Alex Gamburd.
\newblock Uniform expansion bounds for {C}ayley graphs of {${\rm SL}_2(\Bbb
  F_p)$}.
\newblock {\em Ann. of Math. (2)}, 167(2):625--642, 2008.

\bibitem{BGSINVENT}
Jean Bourgain, Alex Gamburd, and Peter Sarnak.
\newblock Affine linear sieve, expanders, and sum-product.
\newblock {\em Invent. Math.}, 179(3):559--644, 2010.

\bibitem{BGS2}
Jean Bourgain, Alex Gamburd, and Peter Sarnak.
\newblock Generalization of {S}elberg's {$\frac{3}{16}$} theorem and affine
  sieve.
\newblock {\em Acta Math.}, 207(2):255--290, 2011.

\bibitem{BK1}
Jean Bourgain and Alex Kontorovich.
\newblock On representations of integers in thin subgroups of {${\rm SL}_2(\Bbb
 Z)$}.
\newblock {\em Geom. Funct. Anal.}, 20(5):1144--1174, 2010.

\bibitem{BKANNALS}
Jean Bourgain and Alex Kontorovich.
\newblock On {Z}aremba's conjecture.
\newblock {\em Ann. of Math. (2)}, 180(1):137--196, 2014.

\bibitem{BVARBITRARY}
Jean Bourgain and P{\'e}ter~P. Varj{\'u}.
\newblock Expansion in {$SL_d({\bf Z}/q{\bf Z}),\,q$} arbitrary.
\newblock {\em Invent. Math.}, 188(1):151--173, 2012.




\bibitem{DOLG}
Dmitry Dolgopyat.
\newblock On decay of correlations in {A}nosov flows.
\newblock {\em Ann. of Math. (2)}, 147(2):357--390, 1998.


\bibitem{FK14}
Dmitrii Frolenkov and Igor D. Kan.
\newblock A strengthening of a theorem of {B}ourgain-{K}ontorovich {II}.
\newblock {\em Mosc. J. Comb. Number Theory}, 4(1):78--117, 2014.

\bibitem{GAMBURDGAP}
Alex Gamburd.
\newblock On the spectral gap for infinite index ``congruence'' subgroups of
  {${\rm SL}_2(\bold Z)$}.
\newblock {\em Israel J. Math.}, 127:157--200, 2002.

\bibitem{HUANG}
ShinnYih Huang.
\newblock An improvement to {Z}aremba's conjecture.
\newblock {\em Geom. Funct. Anal.}, 25(3):860--914, 2015.

\bibitem{KONTOROVICH}
Alex Kontorovich.
\newblock From {A}pollonius to {Z}aremba: {L}ocal-global phenomena in thin orbits.
\newblock {\em Bull. Amer. Soc.}, 50:187--228, 2013.

\bibitem{LALLEYSYMB}
Steven~P. Lalley.
\newblock Renewal theorems in symbolic dynamics, with applications to geodesic
  flows, non-{E}uclidean tessellations and their fractal limits.
\newblock {\em Acta Math.}, 163(1-2):1--55, 1989.

\bibitem{LAXPHILLIPS}
Peter~D. Lax and Ralph~S. Phillips.
\newblock The asymptotic distribution of lattice points in {E}uclidean and
  non-{E}uclidean spaces.
\newblock {\em J. Funct. Anal.}, 46(3):280--350, 1982.



\bibitem{MO}
Amir Mohammadi and Hee Oh.
\newblock Matrix coefficients, counting and primes for orbits of geometrically
  finite groups.
\newblock {\em J. Eur. Math. Soc. (JEMS)}, 17(4):837--897, 2015.

\bibitem{NAUD}
Fr{\'e}d{\'e}ric Naud.
\newblock Expanding maps on {C}antor sets and analytic continuation of zeta
  functions.
\newblock {\em Ann. Sci. \'Ecole Norm. Sup. (4)}, 38(1):116--153, 2005.

\bibitem{OW}
Hee Oh and Dale Winter.
\newblock Uniform exponential mixing and resonance free regions for convex cocompact congruence
  subgroups of $\mathrm{SL}_2(\mathbb{Z})$.
\newblock {\em Journal of the AMS}, Vol 29: 1069--1115, (2016)

\bibitem{PPAST}
William Parry and Mark Pollicott.
\newblock Zeta functions and the periodic orbit structure of hyperbolic
  dynamics.
\newblock {\em Ast\'erisque}, (187-188):268, 1990.

\bibitem{PATTERSON}
S.~J. Patterson.
\newblock The limit set of a {F}uchsian group.
\newblock {\em Acta Math.}, 136(3-4):241--273, 1976.

\bibitem{RUELLEFREDHOLM}
David Ruelle.
\newblock An extension of the theory of {F}redholm determinants.
\newblock {\em Inst. Hautes \'Etudes Sci. Publ. Math.}, (72):175--193 (1991),
  1990.


\bibitem{SELBERGCOEF}
Atle Selberg.
\newblock On the estimation of {F}ourier coefficients of modular forms.
\newblock In {\em Proc. {S}ympos. {P}ure {M}ath., {V}ol. {VIII}}, pages 1--15.
 Amer. Math. Soc., Providence, R.I., 1965.

\bibitem{SERIES81}
Caroline Series.
\newblock The infinite word problem and limit sets in {F}uchsian groups.
\newblock {\em Ergodic Theory Dynamical Systems}, 1(3):337--360 (1982), 1981.



\bibitem{ZAREMBA}
S.~K. Zaremba.
\newblock La m\'ethode des ``bons treillis'' pour le calcul des int\'egrales
  multiples.
\newblock In {\em Applications of number theory to numerical analysis ({P}roc.
  {S}ympos., {U}niv. {M}ontreal, {M}ontreal, {Q}ue., 1971)}, pages 39--119.
  Academic Press, New York, 1972.

\end{thebibliography}

\end{document}